\documentclass[11pt,reqno]{amsart}
\usepackage[margin=1in]{geometry}
\usepackage{adjustbox}
\usepackage{amsmath,amsfonts,amssymb,amsthm}
\usepackage{mathtools}
\usepackage{commath}

\usepackage{graphicx} %package to manage images
\graphicspath{ {images/} }
\usepackage[rightcaption]{sidecap}
\usepackage{caption}
\usepackage{subcaption}
\usepackage{epsfig}

\usepackage{siunitx}
\usepackage{algorithm}
\usepackage{algpseudocode}

\usepackage[utf8]{inputenc}
\usepackage{dirtytalk}
\usepackage[english]{babel}

\usepackage{physics}
\newtheorem{theorem}{Theorem}[section]
\newtheorem{lemma}[theorem]{Lemma}

\newtheorem{corollary}[theorem]{Corollary}

\newenvironment{definition}[1][Definition]{\begin{trivlist}
\item[\hskip \labelsep {\bfseries #1}]}{\end{trivlist}}
\newenvironment{example}[1][Example]{\begin{trivlist}
\item[\hskip \labelsep {\bfseries #1}]}{\end{trivlist}}

\usepackage{bbm}
\newcommand{\defeq}{\vcentcolon=}

\newcommand{\diag}{\mathop{\mathrm{diag}}}
\newcommand{\card}{\mathop{\mathrm{card}}}

\newcommand{\poly}{\mathop{\mathrm{poly}}}

\usepackage{hyperref,xcolor}
\hypersetup{
     colorlinks,
     linkcolor={red!50!black},
     citecolor={blue!50!black},
     urlcolor={blue!80!black}
}

\begin{document}
\title{Rank-Constrained Hyperbolic Programming}
\author{Zhen Dai}
\author{Lek-Heng Lim}
\maketitle

\begin{abstract}
We extend rank-constrained optimization to general hyperbolic programs (HP) using the notion of matroid rank. For LP and SDP respectively, this reduces to sparsity-constrained LP and rank-constrained SDP that are 
already well-studied. But for QCQP and SOCP, we obtain new interesting optimization problems. For example, rank-constrained SOCP includes weighted Max-Cut and nonconvex QP as special cases, and dropping the rank constraints yield 
the standard SOCP-relaxations of these problems. We will show (i) how to do rank reduction for SOCP and QCQP, (ii) that rank-constrained SOCP and rank-constrained QCQP are NP-hard, and (iii) an improved result for rank-constrained SDP showing that if the number of constraints is $m$ and the rank constraint is less than $2^{1/2-\epsilon} \sqrt{m}$ for some $\epsilon>0$, then the problem is NP-hard.
We will also study sparsity-constrained HP and extend results on LP sparsification to SOCP and QCQP. In particular, we show that there always exist (a) a solution to SOCP of cardinality at most twice the number of constraints and (b) a solution to QCQP of cardinality at most the sum of the number of linear constraints and the sum of the rank of the matrices in the quadratic constraints; and both (a) and (b) can be found efficiently.
\end{abstract}

\section{Introduction}
In this paper, we study rank-constrained and sparsity-constrained hyperbolic programming (HP). Specifically, we consider four types of HP: linear programming (LP), quadratically constrained quadratic program (QCQP), second order cone programming (SOCP), and semidefinite programming (SDP).

Rank-constrained SDP occurs frequently in combinatorial optimization \cite{Anjos2002,Lim2017,Luo2010}. It is well-known that Max-Cut could be viewed as a rank-constrained SDP and dropping this rank constraint yields the standard SDP-relaxation of Max-Cut \cite{Anjos2002}. Thus, it is natural to consider when we can get a solution to SDP of small rank. In \cite{Barvinok1995, LemonSoYe2015, Pataki1998}, it is shown that every feasible SDP with $m$ linear constraints always has a solution of rank at most $(\sqrt{1 + 8m} - 1)/2$. Furthermore, this low-rank solution can be found in polynomial time by first solving the SDP and then run a rank reduction algorithm proposed in \cite{Barvinok1995, LemonSoYe2015, Pataki1998}. Specifically, a rank reduction algorithm takes a solution to SDP as input and outputs another low-rank solution to this SDP. This result implies that rank-constrained SDP is polynomial time solvable for any rank constraint that is at least $(\sqrt{1 + 8m} - 1)/2$.

Parallel to SDP rank reduction, LP sparsification is studied in \cite{Cara1907, LemonSoYe2015, Ye2008}. For any feasible LP with $m$ linear constraints, there always exists a solution of cardinality at most $m$. Moreover, this low-cardinality solution can be found in polynomial time by fist solving the LP and then run a LP sparsification algorithm \cite{Cara1907, LemonSoYe2015, Ye2008}. Specifically, a LP sparsification algorithm takes a solution to LP as input and outputs another sparse solution.

For rank-constrained problems, we use rank in HP \cite{Brand2011, Renegar2006} to define rank in LP, QCQP, SOCP, and SDP, by viewing them as special cases of HP. For each of these problems, we study the corresponding rank-constrained problem in two ways. First, we give a polynomial time rank-reduction algorithm to show that it is always possible to get a solution of "small" rank provided that the problem is feasible. Second, we consider the complexity of these rank-constrained problems. Under certain conditions, we will show that rank-constrained LP is polynomial time solvable and rank-constrained QCQP and rank-constrained SOCP are both NP-hard. For rank-constrained SDP with $m$ linear constraints, we consider rank constraint $r(m)$, that is a function of $m$. Then, we show that the complexity of rank-constrained SDP changes as $r(m)$ passes through $\sqrt{2m}$. In particular, rank-constrained SDP is NP-hard when $r(m) \ll \sqrt{2m}$ and polynomial time solvable when $r(m) \gg \sqrt{2m}$.

For sparsity-constrained problems, we extend the results in LP sparsification \cite{Cara1907, LemonSoYe2015, Ye2008} to QCQP and SOCP. Previous results show that every feasible LP with $m$ linear constraints has a solution of cardinality at most $m$, and we can find such a solution in polynomial time \cite{Cara1907, LemonSoYe2015, Ye2008}. In addition, there are examples of LP with $m$ constraints whose solutions have cardinality at least $m$ \cite{Cara1907, LemonSoYe2015, Ye2008}, which shows that this LP sparsification result cannot be improved without further assumptions. We extend this result to QCQP and SOCP and show that our results cannot be improved without further assumptions.

%We will discuss both sparsity-constrained and rank-constrained hyperbolic programming. While the sparsity of a vector $v$ only depends on $v$ itself, the rank of a vector depends on which optimization problem we are considering. In particular, as we shall see, the rank of $v$ in LP is different from the rank of $v$ in SDP. Specifically, we use the notion of rank of a vector with respect to a hyperbolic polynomial as defined in \cite{Brand2011, Renegar2006}. Then we show that each of the optimization problems LP, SOCP, QCQP, and SDP corresponds to a different hyperbolic polynomial. Under this notion of rank, the rank of vector $v$ in LP is exactly its sparsity $\lVert v \rVert_0$ and the rank of a matrix $A \in \mathbb{R}^{n \times n}$ (viewed as a vector in $\mathbb{R}^{n^2}$) is its usual rank $\rank(A)$. Then we will show NP-hardness results of rank constrained problems. As we shall see, the rank constrained LP and QCQP are both polynomial time solvable and the rank constrained SDP and SOCP are both NP-hard. In particular, the weighted Max-Cut problem can be written as either rank constrained SDP \cite{Anjos2002} or rank constrained SOCP. This motivates us to consider rank reduction algorithms for SOCP.

\subsection{Further related works}
Rank for Lorentz cone has been studied through the lens of Euclidean Jordan algebra \cite{symmetric_cone_analysis,Jordan_rank,SDP_handbook}. The definition of rank for points in a Lorentz cone in \cite{symmetric_cone_analysis,Jordan_rank} is the same as the rank for points in SOCP in our work when there is only one second order cone constraint. In this case, the rank estimation theorem in \cite{Jordan_rank} gives the same result as the SOCP rank reduction result in our work.

In addition, our results on SOCP rank reduction can also be deduced from \cite{SDPgeometry,SDP_handbook}. Specifically, the author gives an algorithm which constructs an extreme point solution from any starting solution \cite{SDPgeometry,SDP_handbook} for conic LP problems. Applying this algorithm to SOCP gives a solution of small rank.

%In \cite{NLV2013}, the authors studied a similar complexity problem. They consider the decision problem of completing symmetric matrices with all ones diagonal to positive semidefinite matrix of restricted rank.

%the author characterizes properties of extreme points for conic LP. The author also gives an algorithm which constructs an extreme point solution from any starting solution. Some of our results on HP rank reduction can also be deduced from results in \cite{SDPgeometry,SDP_handbook}.

\section{Rank-Constrained SDP}
In this section, we study Semidefinite Programming (SDP) with rank constraint:
\begin{equation} \label{eq:sdp_rank}
\begin{split}
\underset{X \in \mathbb{S}^n}{\textrm{minimize}} 
\hspace{2mm} & \tr(AX) \\
\textrm{subject to}
\hspace{2mm} & \tr(A_i X) = b_i, \; i = 1, \dots, m; \\
&  X \geq 0; \\
& \rank(X) \leq r(m),
\end{split}
\end{equation}
where $A,A_1,\dots,A_m \in \mathbb{S}^n$, $b_1, \dots, b_m \in \mathbb{R}$, and $r(m)$ is a function in $m$. In this section, we begin with some examples of rank-constrained SDP. Then, we study the condition under which SDP is NP-hard. We will show that there is a phase transition in the complexity of rank-constrained SDP when $r(m)$ passes through $\sqrt{2m}$.

\subsection{Examples of rank-constrained SDP}
Rank-constrained SDP appears in many combinatorial problems such as weighted Max-Cut, clique number, and stability number. In the following, we formulate these combinatorial problems in terms of rank-constrained SDP.

%To motivate this problem, we give some examples of combinatorial problems that can be written as rank constrained SDP. 

%In the following let $G = (V,E)$ be the underlying graph and let $n = |V|$ be the number of vertices. We begin with the max cut problem \cite{Anjos2002}.

\begin{example}
Consider a graph $G = (V,E)$ with vertex set $V$ and edge set $E$. Let $w: E \longrightarrow \mathbb{R}$ be a weight function on $G$. Without loss of generality, we might assume that $G$ is a complete graph (i.e. $(i,j) \in E$ for all $i,j \in V$) and some of the edges have zero weight (i.e. $w(e) = 0$ for some $e \in E$). In weighted Max-Cut problem, the goal is to find a partition of $V = V_1 \sqcup V_2$, that maximizes the sum of the weights on edges whose endpoints lie in different portions of the partition. To be specific, we want to solve the following problem:
\begin{equation*}
    \begin{split}
        \underset{V = V_1 \sqcup V_2}{\textrm{maximize}}
        \hspace{2mm} & \sum_{i \in V_1, j \in V_2} w(i,j).
    \end{split}
\end{equation*}
When $w(i,j) \in \{0,1\}$ for all $i,j \in V$, we call this problem the unweighted Max-Cut problem.

For simplicity, we identify $V$ with $[n] \defeq \{1,2,\dots,n\}$. We define the weight matrix $W \in \mathbb{R}^{n \times n}$ by
\begin{equation*}
    W[i,j] = \begin{cases}
         w(i,j)/4 & \textrm{if} \quad i \neq j \\
         -\sum_{k \neq i}w(i,k)/4 & \textrm{if} \quad i = j.
    \end{cases}
\end{equation*}
Then weighted Max-Cut is equivalent to the following rank constrained SDP \cite{Anjos2002}:

\begin{equation} \label{eq:1}
\begin{split}
\underset{X \in \mathbb{S}^n}{\textrm{minimize}} 
\hspace{2mm} & \tr(WX) \\
\textrm{subject to}
\hspace{2mm} & X_{ii} = 1, \; i = 1, \dots, n; \\
&  X \geq 0; \\
& \rank(X) \leq 1.
\end{split}
\end{equation}
\end{example}

%Next, we consider the problem of computing clique number $\omega(G)$ and stability number $\alpha(G)$ of a graph $G$. 

\begin{example}
Next, we consider the problem of computing clique number. Given a graph $G$, a clique in $G$ is a subgraph $H$ of $G$ such that any two distinct vertices in $H$ are adjacent (i.e. there is an edge between them in $G$). The clique number $\omega(G)$ of $G$ is the maximum number of vertices in a clique in $G$.

By results in \cite{Lim2017},
\[
1 - \frac{1}{\omega(G)} = 2\underset{x \in \Delta^n}{\max} \sum_{(i,j) \in E} x_ix_j,
\]
where $\Delta^n = \{x \in \mathbb{R}^n: x_1 + \dots + x_n = 1, x_i \geq 0\}$ is the unit simplex in $\mathbb{R}^n$. Thus, to compute the clique number, it suffices to solve the follow QP:
\[
\begin{split}
\underset{x \in \mathbb{R}^n}{\textrm{minimize}} 
\hspace{2mm} & \sum_{(i,j) \in E} -x_ix_j \\
\textrm{subject to}
\hspace{2mm} & x_i \geq 0, \; i = 1, \dots, n; \\
&  \sum_{i = 1}^n x_i = 1.
\end{split}
\]
By results in \cite{Luo2010}, we can convert this QP into a rank constrained SDP. We first homogenize it to obtain the following QP:
\[
\begin{split}
\underset{x \in \mathbb{R}^n, t \in \mathbb{R}}{\textrm{minimize}} 
\hspace{2mm} & \sum_{(i,j) \in E} -x_ix_j \\
\textrm{subject to}
\hspace{2mm} & tx_i \geq 0, \; i = 1, \dots, n; \\
&  \sum_{i = 1}^n tx_i = 1; \\
&   t^2 = 1.
\end{split}
\]
This is clearly equivalent to the original QP by substituting $tx$ for $x$. Let $A \in \mathbb{R}^{(n+1) \times (n+1)}$ be such that
\[
    A[i,j] = \begin{cases} -1 \hspace{2mm} &\textrm{if}  \hspace{2mm} i\leq n, j \leq n, (i,j) \in E, \\
    0  \hspace{2mm} &\textrm{otherwise}.
\end{cases}
\]
Then the homogenized QP is equivalent to the following rank constrained SDP \cite{Luo2010}:
\[
\begin{split}
\underset{X \in S^{n+1}}{\textrm{minimize}} 
\hspace{2mm} & \tr(AX) \\
\textrm{subject to}
\hspace{2mm} & X_{i(n+1)} \geq 0, \; i = 1, \dots, n; \\
& \sum_{i = 1}^n X_{i(n+1)} = 1; \\
& X_{(n+1)(n+1)} = 1; \\
&  X \geq 0; \\
& \rank(X) \leq 1.
\end{split}
\]
For any solution $X$, $\rank(X) \neq 0$ since $X_{(n+1)(n+1)} = 1$. Thus, $X = vv^T$ for some $v \in \mathbb{R}^{n+1}$. Then $(x,t) = v$ is a solution to the homogenized QP.
\end{example}

\begin{example}
Next, we consider the problem of computing stability number. Given a graph $G$, the stability number $\alpha(G)$ of $G$ is the maximum number of vertices in $G$, of which no two are adjacent (i.e. there is no edge between them in $G$).

By results in \cite{Lim2017},
\[
1 - \frac{1}{\alpha(G)} = 2\underset{x \in \Delta^n}{\max} \sum_{(i,j) \notin E} x_ix_j,
\]
where $\Delta^n = \{x \in \mathbb{R}^n: x_1 + \dots + x_n = 1, x_i \geq 0\}$ is the unit simplex in $\mathbb{R}^n$. Thus, to compute the clique number, it suffices to solve the follow QP:
\[
\begin{split}
\underset{x \in \mathbb{R}^n}{\textrm{minimize}} 
\hspace{2mm} & \sum_{(i,j) \notin E} -x_ix_j \\
\textrm{subject to}
\hspace{2mm} & x_i \geq 0, \; i = 1, \dots, n; \\
&  \sum_{i = 1}^n x_i = 1.
\end{split}
\]
Let $B \in \mathbb{R}^{(n+1) \times (n+1)}$ be such that
\[
    B[i,j] = \begin{cases} -1 \hspace{2mm} &\textrm{if}  \hspace{2mm} i\leq n, j \leq n, (i,j) \notin E, \\
    0  \hspace{2mm} &\textrm{otherwise}.
\end{cases}
\]
By the same argument as we used in the clique number example, this QP is equivalent to the following rank constrained SDP \cite{Luo2010}:
\[
\begin{split}
\underset{X \in S^{n+1}}{\textrm{minimize}} 
\hspace{2mm} & \tr(BX) \\
\textrm{subject to}
\hspace{2mm} & X_{i(n+1)} \geq 0, \; i = 1, \dots, n; \\
& \sum_{i = 1}^n X_{i(n+1)} = 1; \\
& X_{(n+1)(n+1)} = 1; \\
&  X \geq 0; \\
& \rank(X) \leq 1.
\end{split}
\]
\end{example}

\subsection{Complexity of rank-constrained SDP}
In this section, we give the condition under which rank constrained SDP is NP-hard. Recall that rank-constrained SDP is formulated as:
\begin{equation*}
\begin{split}
\underset{X \in \mathbb{S}^n}{\textrm{minimize}} 
\hspace{2mm} & \tr(AX) \\
\textrm{subject to}
\hspace{2mm} & \tr(A_i X) = b_i, \; i = 1, \dots, m; \\
&  X \geq 0; \\
& \rank(X) \leq r(m),
\end{split}
\end{equation*}
where $A,A_1,\dots,A_m \in \mathbb{S}^n$, $b_1, \dots, b_m \in \mathbb{R}$, and $r(m)$ is a function in $m$. In last section, we see that weighted Max-Cut can be formulated as a rank-constrained SDP with $r(m) = 1$. As a result, rank-constrained SDP is NP-hard if $r(m) = 1$ for all $m$. On the other hand, for any feasible SDP, we could first solve the vanilla SDP (i.e. without rank constraint) and then run a rank reduction algorithm to find another optimal solution of rank at most $(\sqrt{1 + 8m} - 1)/2$ \cite{Barvinok1995, LemonSoYe2015, Pataki1998}. Thus, if $r(m) \geq (\sqrt{1 + 8m} - 1)/2$ for all $m$, then we can always solve the rank-constrained SDP by the procedure we just described. Assuming that we can compute real numbers exactly, solving the vanilla SDP (without rank constraint) and running the rank reduction algorithm can both be done in polynomial time \cite{Barvinok1995, LemonSoYe2015, Pataki1998}. Roughly speaking, when $r(m) \gg \sqrt{2m}$, rank-constrained SDP is polynomial time solvable. In the following result, we show that when $r(m) \ll \sqrt{2m}$, rank-constrained SDP is NP-hard.

\begin{theorem} \label{thm:00}
Let $A,A_1,\dots,A_m \in \mathbb{S}^n$, $b_1, \dots, b_m \in \mathbb{R}$, and $r:\mathbb{Z}^+ \longrightarrow \mathbb{Z}^+$ be given. Suppose that there exist constants $M, \epsilon>0$, such that
\begin{equation} \label{eq:rank}
    r(m) < 2^{1/2 - \epsilon} \sqrt{m}, \quad \textrm{for all} \quad m \geq M.
\end{equation}
Then, the rank-constrained SDP
\begin{equation*}
\begin{split}
\underset{X \in \mathbb{S}^n}{\operatorname{minimize}} 
\hspace{2mm} & \tr(AX) \\
\operatorname{subject \hspace{1mm} to}
\hspace{2mm} & \tr(A_i X) = b_i, \; i = 1, \dots, m; \\
&  X \geq 0; \\
& \rank(X) \leq r(m),
\end{split}
\end{equation*}
is NP-hard.
\end{theorem}

The above result, together with results of SDP rank reduction \cite{Barvinok1995, LemonSoYe2015, Pataki1998}, show that there is a phase transition in the complexity of rank-constrained SDP as $r(m)$ passes through $\sqrt{2m}$.

Before giving the formal proof, we explain the main ideas of this proof. To begin with, consider the case $r(m) = 2$. Recall that weighted Max-Cut is equivalent to the following rank-constrained SDP:
\begin{equation} \label{eq:maxcut}
\begin{split}
\underset{X \in \mathbb{S}^n}{\textrm{minimize}} 
\hspace{2mm} & \tr(WX) \\
\textrm{subject to}
\hspace{2mm} & X_{ii} = 1, \; i = 1, \dots, n; \\
&  X \geq 0; \\
& \rank(X) \leq 1.
\end{split}
\end{equation}
Now, we transform the above rank-constrained SDP into a new rank-constrained SDP, whose rank constraint is two.
Let 
\begin{equation*}
    W' = 
\begin{bsmallmatrix}
  0 & 0\\ 
  0 & W
\end{bsmallmatrix} \in \mathbb{R}^{(n+1) \times (n+1)},
\end{equation*}
where $0$ denotes vectors whose entries are zeros.
Then, consider the following rank-constrained SDP:
\begin{equation} \label{eq:maxcut2}
\begin{split}
\underset{X' \in S^{n+1}}{\textrm{minimize}} \hspace{2mm} & \tr(W'X') \\
\textrm{subject to}
\hspace{2mm }& X'_{ii} = 1, \; i = 1, \dots, n+1; \\
& X'_{1j} = 0, \; j = 2, \dots, n+1; \\
& X' \geq 0; \\
& \rank(X') \leq 2.
\end{split}
\end{equation}
Note that any solution $X'$ to the above rank-constrained SDP must have the form
\begin{equation*}
    X' = 
\begin{bsmallmatrix}
  1 & 0\\ 
  0 & X
\end{bsmallmatrix} \in \mathbb{R}^{(n+1) \times (n+1)},
\end{equation*}
where $0$ denotes vectors whose entries are zeros.
The rank-constrained SDP in \eqref{eq:maxcut2} is equivalent to the one in \eqref{eq:maxcut} since $\rank(X') = \rank(X) + 1$.

By applying this ``rank increment'' technique, we can show that rank-constrained SDP with constant rank constraint (i.e. $r(m) = r$ for some constant $r$) is NP-hard.
To get the NP-hardness result for $r$, we need $m = rn + r(r-1)/2 = r(n-1) + r(r+1)/2$ many linear constraints. In other words,
\begin{equation*}
n = \Biggl\lfloor \Biggl(m - \frac{r(r+1)}{2} \Biggr) / r \Biggr\rfloor + 1.
\end{equation*}
Thus, we need $m > r(r+1)/2$. Roughly speaking, $r \ll \sqrt{2m}$. In addition, note that as long as $r \ll \sqrt{2m}$, $n = \Omega(\sqrt{m}) = \Omega(r)$. Thus, the new dimension of the problem $n+r-1$ is polynomial in the original dimension $n$. Thus, any polynomial time algorithm on this transformed problem translates to a polynomial time algorithm on the original problem.

\begin{proof}
We begin with the special case that $r(m)$ is non-deceasing (i.e. $r(m+1) \geq r(m)$ for all $m$). Then, we will drop this additional assumption.

\pmb{Special Case:} \par
Our goal is to reduce weighted Max-Cut to rank-constrained SDP. Suppose that there is a polynomial time algorithm $\mathcal{A}$ for rank-constrained SDP with $r(m)$ satisfying equation \eqref{eq:rank}. Then, we will show that we can use this algorithm to solve weighted Max-Cut in polynomial time. Let
\begin{equation*}
    \phi(m) = \Biggl\lfloor \bigg(m - \frac{r(m)(r(m)+1)}{2} \bigg) \bigg/ r(m) \Biggr\rfloor + 1.
\end{equation*}
Given an input graph $G$ with $n$ nodes and a weight function $w$, we consider two cases $(i) n \geq C$ and $(ii) n < C$ separately, where $C$ is some constant which only depends on $\epsilon$ and $M$. We will pick $C$ later in the proof but it can be determined before we receive the input of weighted Max-Cut. 

If $n \geq C$, our algorithm works as follows:
\begin{enumerate}
    \item find $m$ such that $n = \phi(m)$;
    \item construct a rank-constrained SDP with $m$ linear constrains that is equivalent to weighted Max-Cut on the weighted graph $(G,w)$;
    \item solve this rank-constrained SDP using algorithm $\mathcal{A}$.
\end{enumerate}
If $n<C$, we use brute force to solve the problem. Now, we discuss each step in detail.

\pmb{$(1)$}
We begin with some observations on $\phi(m)$. For $m \geq M$,
\begin{equation} \label{pfrank:01}
\begin{split}
    \phi(m) & \geq \bigg(m - \frac{r(m)(r(m)+1)}{2} \bigg) \bigg/ r(m) \\
    & \stackrel{(a)} \geq \bigg(m - 2^{-2\epsilon}m - 2^{-1/2-\epsilon} \sqrt{m} \bigg) \bigg/ r(m) \\
    & \stackrel{(b)} \geq \delta \sqrt{m} \quad \textrm{for some} \hspace{1mm} \delta>0, \hspace{1mm} \textrm{for all} \hspace{1mm} m \geq M', \hspace{1mm} \textrm{for some} \hspace{1mm} M'>0,
\end{split}
\end{equation}
where $\delta$ and $M'$ only depend on $\epsilon$, and we used equation \eqref{eq:rank} in $(a)$ and $(b)$. Now, we pick $C = max(M,M')+1$. Recall that we only consider input graph $G$ with $n \geq C$ nodes.

Since $r(m)$ is non-decreasing, for all $m \geq C-1$,
\begin{equation} \label{pfrank:02}
\begin{split}
    \phi(m+1) - \phi(m) & = \Biggl\lfloor \bigg(m+1 - \frac{r(m+1)(r(m+1)+1)}{2} \bigg) \bigg/ r(m+1)\Biggr\rfloor - \Biggl\lfloor \bigg(m - \frac{r(m)(r(m)+1)}{2} \bigg) \bigg/ r(m)\Biggr\rfloor \\
    & \leq \Biggl\lfloor \bigg(m+1 - \frac{r(m)(r(m)+1)}{2} \bigg) \bigg/ r(m)\Biggr\rfloor - \Biggl\lfloor \bigg(m - \frac{r(m)(r(m)+1)}{2} \bigg) \bigg/ r(m)\Biggr\rfloor \\
    & \leq 1.
\end{split}
\end{equation}

Note that
\begin{equation} \label{pfrank:03}
    \phi(n-1) \leq n-1+1 = n.
\end{equation}
In addition, by equation \eqref{pfrank:01},
\begin{equation} \label{pfrank:04}
    \phi(\lceil n/\delta \rceil^2) \geq n.
\end{equation}
By equation \eqref{pfrank:02}, \eqref{pfrank:03}, and \eqref{pfrank:04}, there exists some $m \in [n-1,\lceil n/\delta \rceil^2]$ such that $\phi(m) = n$. By computing $\phi(m)$ from $m = n-1$ to $m = \lceil n/\delta \rceil^2$, we can find $m$ with $\phi(m) = n$ in $\mathcal{O}(n^2)$ time.

\pmb{$(2)$}
Recall that weighted Max-Cut is equivalent to the following rank constrained SDP \cite{Anjos2002}:
\begin{equation} \label{pfsdp:1}
\begin{split}
\underset{X \in \mathbb{S}^n}{\textrm{minimize}} 
\hspace{2mm} & \tr(WX) \\
\textrm{subject to}
\hspace{2mm} & X_{ii} = 1, \; i = 1, \dots, n; \\
&  X \geq 0; \\
& \rank(X) \leq 1,
\end{split}
\end{equation}
where $W \in \mathbb{R}^{n \times n}$ is the weight matrix induced by the weighted graph $(G,w)$:
\begin{equation*}
    W[i,j] = \begin{cases}
         w(i,j)/4 & \textrm{if} \quad i \neq j \\
         -\sum_{k \neq i}w(i,k)/4 & \textrm{if} \quad i = j.
    \end{cases}
\end{equation*}
Now, we will construct an equivalent rank-constrained SDP of dimension $n+r(m)-1$. Let
\begin{equation*}
    W' = 
    \begin{bmatrix}
    0 & 0 \\
    0 & W
    \end{bmatrix} \in \mathbb{R}^{(n+r(m)-1) \times (n+r(m)-1)},
\end{equation*}
where $0$ denotes a matrix whose entries are zeroes. Then, we claim that the rank-constrained SDP in \eqref{pfsdp:1} is equivalent to the following rank-constrained SDP:
\begin{equation} \label{pfsdp:2}
\begin{split}
\underset{X' \in \mathbb{S}^{n+r(m)-1}}{\textrm{minimize}} 
\hspace{2mm} & \tr(W'X') \\
\textrm{subject to}
\hspace{2mm} & X'_{ii} = 1, \; i = 1, \dots, n+r(m)-1; \\
& X'_{ij} = 0, \; j = i+1, \dots, n+r(m)-1; \; i = 1, \dots, r(m)-1; \\
&  X' \geq 0; \\
& \rank(X') \leq r(m),
\end{split}
\end{equation}
To see this, note that any solution $X'$ to the rank-constrained SDP in \eqref{pfsdp:2} must have the form
\begin{equation*}
    X' = 
    \begin{bmatrix}
    I_{r(m)-1} & 0 \\
    0 & X
    \end{bmatrix} \in \mathbb{R}^{(n+r(m)-1) \times (n+r(m)-1)},
\end{equation*}
where $I_{r(m)-1} \in \mathbb{R}^{(r(m)-1) \times (r(m)-1)}$ is the identity matrix, $X \in \mathbb{S}^n$, and $0$ denotes a matrix whose entries are zeroes. Thus, $\rank(X') \leq r(m)$ if and only if $\rank(X) \leq 1$ and $X' \geq 0$ if and only if $X \geq 0$. In addition, $\tr(W'X') = \tr(WX)$. Thus, the rank-constrained SDP in \eqref{pfsdp:2} is equivalent to the one in \eqref{pfsdp:1}. Thus, in order to solve weighted Max-Cut, it suffices to solve the rank-constrained SDP in \eqref{pfsdp:2}. Note that the rank-constrained SDP in \eqref{pfsdp:2} has 
\begin{equation} \label{pfsdp:count_constraint}
    \sum_{i = n}^{n+r(m)-1}i = \frac{(2n+r(m)-1)r(m)}{2} = (n-1)r(m) + \frac{r(m)(r(m)+1)}{2}
\end{equation}
many linear constraints. Since $n = \phi(m)$,
\begin{equation} \label{pfsdp:count_constraint2}
    (n-1)r(m) + \frac{r(m)(r(m)+1)}{2} \leq m.
\end{equation}
By adding superfluous linear constraints to \eqref{pfsdp:2}, we get a rank-constrained SDP which is equivalent to \eqref{pfsdp:2} and has exactly $m$ linear constraints.

\pmb{$(3)$}
Finally, we can apply algorithm $\mathcal{A}$ to solve this rank-constrained SDP in $\poly(n+r(m)-1)$ time. Since we search for $m$ in $[n-1,\lceil n/\delta \rceil^2]$ in step $(1)$, $m \geq n-1$. Since $n \geq C$, $m \geq C-1 \geq M$. Thus,
\begin{equation} \label{pfsdp:31}
    r(m) < 2^{1/2 - \epsilon} \sqrt{m}.
\end{equation}
Since $n = \phi(m) \geq \delta \sqrt{m}$ by equation \eqref{pfrank:01},
\begin{equation*}
    r(m) = \mathcal{O}(n),
\end{equation*}
by equation \eqref{pfsdp:31}. Thus, algorithm $\mathcal{A}$ solves the rank-constrained SDP in $\poly(n+r(m)-1) = \poly(n)$ time.

\pmb{Case when $n<C$}
If the number of nodes $n$ is less than $C$, we solve weighted Max-Cut by brute force. We simply trying each of the $2^n$ possible partitions of the vertex set and compute the value of the cut in each case. Then, we take the maximum one. Since $n<C$, this takes at most $\mathcal{O}(2^C)$ time, which is a constant. Thus, overall, the algorithm runs in $\poly(n) + \mathcal{O}(2^C) = \poly(n)$ time.

\pmb{General Case:} \par
Now, we drop the assumption that $r(m+1) \geq r(m)$ for all $m$. Note that the only place we used this assumption in the proof of special case is to show that
\begin{equation*}
    \phi(m+1) - \phi(m) \leq 1.
\end{equation*}
The way to avoid using this assumption is to change the definition of $\phi(m)$. First, note that in the proof of the special case, we used $r(m)$ only when $m \geq M$. Thus, we can assume without loss of generality that
\begin{equation} \label{pfgen:newr}
    r(m) < 2^{1/2 - \epsilon} \sqrt{m}, \quad \textrm{for all} \quad m \in \mathbb{Z}^+.
\end{equation}
Then, define $\Tilde{r}:\mathbb{Z}^+ \longrightarrow \mathbb{Z}^+$ as
\begin{equation} \label{pfgen:newr2}
    \Tilde{r}(m) = \max(r(i):i = 1,\dots,m).
\end{equation}
Clearly, 
\begin{equation} \label{pfgen:eq:02}
    \Tilde{r}(m+1) \geq \Tilde{r}(m), \quad \textrm{for all} \hspace{1mm} m.
\end{equation}
Note that for any $m \in \mathbb{Z}^+$,
\begin{equation} \label{pfgen:eq:01}
\begin{split}
    \Tilde{r}(m) & = r(t) \quad \textrm{for some} \hspace{1mm} t \leq m \\
    & \stackrel{(a)} < 2^{1/2 - \epsilon} \sqrt{t} \\
    & \leq 2^{1/2 - \epsilon} \sqrt{m},
\end{split}
\end{equation}
where we used equation \eqref{pfgen:newr} in $(a)$. Let
\begin{equation} \label{pfgen:newphi}
    \phi(m) = \Biggl\lfloor \bigg(m - \frac{\Tilde{r}(m)(\Tilde{r}(m)+1)}{2} \bigg) \bigg/ \Tilde{r}(m) \Biggr\rfloor + 1.
\end{equation}
The skeleton of the algorithm is similar as before. Given an input graph $G$ with $n$ nodes and a weight function $w$, we consider two cases $(i) n \geq C$ and $(ii) n < C$ separately, where $C$ is some constant which only depends on $\epsilon$ and $M$. We will pick $C$ later in the proof but it can be determined before we receive the input of weighted Max-Cut. 

If $n \geq C$, our algorithm works as follows:
\begin{enumerate}
    \item find $m$ such that $n = \phi(m)$;
    \item construct a rank-constrained SDP with $m$ linear constrains that is equivalent to weighted Max-Cut on the weighted graph $(G,w)$;
    \item solve this rank-constrained SDP using algorithm $\mathcal{A}$.
\end{enumerate}
If $n<C$, we use brute force to solve the problem. Now, we discuss each step in detail.

\pmb{$(1)$} By equations \eqref{pfgen:eq:01}, \eqref{pfgen:eq:02}, and \eqref{pfgen:newphi}, we can get
\begin{equation} \label{pfgen:eq:318}
\begin{split}
    & \phi(m) \geq \delta \sqrt{m}; \\
    & \phi(m+1) - \phi(m) \leq 1; \\
    & \phi(n-1) \leq n; \\
    & \phi(\lceil n/\delta \rceil^2) \geq n,
\end{split}
\end{equation}
in the same way as we did in the special case of the proof. Thus, this step remains unchanged.

\pmb{$(2)$} Once we find $m$ such that $n = \phi(m)$, we consider the rank-constrained SDP in \eqref{pfsdp:2}, in exactly the same way as we did in the special case. Note that we use $r(m)$ instead of $\Tilde{r}(m)$ here. It is important to note that we use $\Tilde{r}(m)$ solely to choose the value of $m$. After that, we only use $r(m)$. Thus, the number of linear constraints in \eqref{pfsdp:2} is exactly the same as we counted in equation \eqref{pfsdp:count_constraint}, which is $(n-1)r(m) + r(m)(r(m)+1)/2$. Then,
\begin{equation}
\begin{split}
    & (n-1)r(m) + r(m)(r(m)+1)/2 \\
    & \stackrel{(a)} \leq (n-1)\Tilde{r}(m) + \Tilde{r}(m)(\Tilde{r}(m)+1)/2 \\
    & \stackrel{(b)} \leq m,
\end{split}
\end{equation}
where we used equation \eqref{pfgen:newr2} in $(a)$ and equation \eqref{pfgen:newphi} in $(b)$. The rest of this step remains unchanged.

\pmb{$(3)$}
Finally, we need to show that $r(m) = \mathcal{O}(n)$. This holds since $n = \phi(m) \geq \delta \sqrt{m}$ by equation \eqref{pfgen:eq:318}, and $r(m) < 2^{1/2 - \epsilon} \sqrt{m}$ by equation \eqref{pfgen:newr}.

The case when $n<C$ is exactly the same as before.
\end{proof}

\begin{corollary} \label{cor:sdp_feasible}
Let $A_1,\dots,A_m \in \mathbb{S}^n$, $b_1, \dots, b_m \in \mathbb{R}$, and $r:\mathbb{Z}^+ \longrightarrow \mathbb{Z}^+$ be given. Suppose that there exist constants $M, \epsilon>0$, such that
\begin{equation*}
    r(m) < 2^{1/2 - \epsilon} \sqrt{m}, \quad \textrm{for all} \quad m \geq M.
\end{equation*}
Then, the rank-constrained SDP feasibility problem:
\begin{equation} \label{eq:sdp_feasibility}
\begin{split}
\operatorname{Find}
\hspace{2mm} & X \in \mathbb{S}^n \\
\operatorname{subject \hspace{1mm} to}
\hspace{2mm} & \tr(A_i X) = b_i, \; i = 1, \dots, m; \\
&  X \geq 0; \\
& \rank(X) \leq r(m),
\end{split}
\end{equation}
is NP-hard.
\end{corollary}

%Note that when $r(m)$ is a constant function, Corollary \ref{cor:sdp_feasible} reduces to the result in \cite{NLV2013}.

\begin{proof}
Suppose that there is a polynomial time algorithm which solves \eqref{eq:sdp_feasibility}. We call this algorithm a feasibility oracle. Then we show that we can also solve the unweighted Max-Cut problem in polynomial time using this feasibility oracle. Let $G = (V,E,w)$ be a weighted graph. Since we are solving the unweighted Max-Cut problem, we may assume that $w(i,j) \in \{0,1\}$ for all $i,j \in V$. Let $n = |V|$ be the number of nodes in $G$.

By the same arguments as in the proof of Theorem \ref{thm:00}, unweighted Max-Cut is equivalent to some rank-constrained SDP
\begin{equation} \label{pf:sdp_fea:01}
\begin{split}
\underset{X \in \mathbb{S}^N}{\textrm{minimize}} 
\hspace{2mm} & \tr(AX) \\
\textrm{subject to}
\hspace{2mm} & \tr(A_i X) = b_i, \; i = 1, \dots, m; \\
&  X \geq 0; \\
& \rank(X) \leq r(m),
\end{split}
\end{equation}
where $N = n+r(m)-1 = \poly(n)$, when $n \geq C$ for some constant $C$. When $n < C$, we can use brute force to find the solution to unweighted Max-Cut problem in $2^{|E|} = \mathcal{O}(2^{C^2}) = \mathcal{O}(1)$ time. Then, it suffices to solve \eqref{pf:sdp_fea:01} using the feasibility oracle. To solve \eqref{pf:sdp_fea:01}, we just write the objective in \eqref{pf:sdp_fea:01} as a linear constraint $\tr(AX) \leq c$ and try for different $c$ by bisection search. For the unweighted Max-Cut problem, we know that the objective is bounded below by $0$ and bounded above by $n^2$. Since we know that the solution to unweighted Max-Cut is an integer, we can find the solution by applying the feasibility oracle $O(\log n)$ times. Thus, we could solve the unweighted Max-Cut in polynomial time. Since unweighted Max-Cut is NP-hard, we are done.
\end{proof}

\section{Sparsity-Constrained Problems}
In this section, we study sparsity-constrained problems. We start with sparsity-constrained Linear Programming (LP):
\begin{equation} \label{def:sclp}
\begin{split}
    \underset{x\in \mathbb{R}^n}{\textrm{minimize}} \hspace{2mm} & c^Tx \\
    \textrm{subject to}  \hspace{2mm} & Ax = b; \\
                      & x \geq 0; \\
                      & \card(x) \leq \kappa,
\end{split}
\end{equation}
where $A \in \mathbb{R}^{m \times n}, b \in \mathbb{R}^m, c \in \mathbb{R}^n$, $\card(x)$ is the cardinality of $x$, which is the number of nonzero entries of $x$, and $\kappa \geq 0$ is a constant. In this case, cardinality is the analogue of rank in SDP. To be specific, if we write LP as SDP such that $x$ is mapped to $X = \diag(x)$, where $\diag(x)$ denotes the diagonal matrix whose diagonal entries are entries of $x$, then $\card(x) = \rank(X)$. Unlike rank-constrained SDP, sparsity constrained LP \eqref{def:sclp} is polynomial time solvable for any constant $\kappa$.

%In the previous section, we considered SDP rank reduction. A special case of it is LP sparsification. In this question, we ask a similar question to LP sparsification and extend this sparsification technique\cite{Cara1907, LemonSoYe2015, Ye2008} to conic LP problems.

%We begin with some notations. Let $n$ be the number of variables in a LP problem. Let $m$ be the number of constraints in a LP problem. Let $x$ be the solution to a LP problem. Let $\card(x)$ denote the cardinality of $x$, which is the number of nonzero entries of $x$. Let $\diag(x)$ denote the diagonal matrix whose diagonal entries are the entries of $x$.

%It is well known that every solvable LP has a solution $x$ whose cardinality is at most $m$ \cite{Cara1907, LemonSoYe2015, Ye2008}. In this case, $\card(x)$ is the analogue of rank in the SDP case. Obviously, if we write LP as SDP such that $x$ is mapped to $X = \diag(x)$, then $\card(x) = \rank(X)$. Now we ask a similar question: for any constant $c$, is LP with the constraint $\card(x) \leq c$ polynomial time solvable? Put it more rigorously: Does there exist a polynomial time algorithm that either gives the optimal solution to the cardinality constrained LP problem or claim that no solution exist? The answer is affirmative and we state it as a proposition.

\begin{theorem} \label{thm:lp:complexity}
Let $A \in \mathbb{R}^{m \times n}, b \in \mathbb{R}^m$, and $c \in \mathbb{R}^n$ be given. Then, for any constant $\kappa \geq 0$, the sparsity-constrained LP problem
\[
\begin{split}
    \underset{x\in \mathbb{R}^n}{\operatorname{minimize}} \hspace{2mm} & c^Tx \\
    \operatorname{subject \hspace{1mm} to}  \hspace{2mm} & Ax = b; \\
                      & x \geq 0; \\
                      & \card(x) \leq \kappa,
\end{split}
\]
is polynomial time solvable.
\end{theorem}

\begin{proof}
To solve this problem, it suffices to solve $\binom{n}{\kappa}$ many LP problems. We encode each LP by a set $S \subset [n]$ of size $n-\kappa$. For each $S$, we set the variable $x_i$ to be $0$ for all $i \in S$. Then we solve the resulting LP defined as follows. Let $c = (c_1,\dots,c_n)^T$ and $A = (a_1,\dots,a_n)$, where $a_i \in \mathbb{R}^m$. Let $c_S$ be the subvector of $c$ obtained by dropping the entries $c_i$ for all $i \in S$. Similarly, let $A_S$ be the submatrix of $A$ obtained by dropping columns $a_i$ for all $i \in S$. Then, the LP corresponding to $S$ is defined as
\[
\begin{split}
    \underset{x' \in \mathbb{R}^{n-\kappa}}{\textrm{minimize}} \hspace{2mm} & c_S^T x' \\
    \textrm{subject to}  \hspace{2mm} & A_S x' = b; \\
                      & x' \geq 0.
\end{split}
\]
If the above LP is solvable, we keep its solution $x_S$ and optimal value $y_S$. If it is not solvable, we do nothing. If none of the $\binom{n}{\kappa}$ LP is solvable, we claim that no solution to the sparsity-constrained LP exists. Otherwise, let $S^*$ be the set such that $y_{S^*}$ is minimal among all $y_S$'s. Then, we output the solution $\Tilde{x}_{S^*}$ defined by 
\[
\begin{cases}
    \Tilde{x}_{S^*}[i] = 0 \hspace{2mm} &\textrm{if}  \hspace{2mm} i \in S^*\\
    \Tilde{x}_{S^*}[i] = x_{S^*}[k_i] \hspace{2mm} &\textrm{if}  \hspace{2mm} i \notin S^*,
\end{cases}
\]
where $i$ is the $k_i$th element in $[n]-S^*$, and the optimal value $y_{S^*}$. Since LP is polynomial time solvable and $\binom{n}{\kappa} \leq n^\kappa$ is polynomial in $n$, this algorithm runs in polynomial time.
\end{proof}

Note that the above result does not contradict the fact that finding minimum-cardinality solution for LP is NP-hard \cite{Garey1979}. The reason is that Theorem \ref{thm:lp:complexity} only applies to constant constraint on cardinality. However, in order to find minimum-cardinality solution for LP, we need to consider cardinality constraints that depend on $n$.

Parallel to the rank reduction results in SDP \cite{Barvinok1995, LemonSoYe2015, Pataki1998}, every feasible LP has a solution $x^*$ whose cardinality is at most $m$ \cite{Cara1907, LemonSoYe2015, Ye2008}.

%Next, we extend the sparsification techniques in LP to conic LP. A conic LP is defined as
%\begin{equation*}
%\begin{split}
%    \underset{x\in \mathbb{R}^n}{\textrm{minimize}} \hspace{2mm} & c^Tx \\
%    \textrm{subject to}  \hspace{2mm} & Ax = b; \\
%                      & x \in C,
%\end{split}
%\end{equation*}
%where $A \in \mathbb{R}^{m \times n}, b \in \mathbb{R}^m, c \in \mathbb{R}^n$, and $C$ is a convex cone. When $C = R_+^n \vcentcolon = \{x \in \mathbb{R}^n: x_i \geq 0, \quad \textrm{for all} \quad i \in [n]\}$, conic LP becomes the vanilla LP. Conic LP also includes SOCP and QCLP. We consider them one by one in the following sections.

Next, we extend the sparsification techniques in LP to Quadratically Constrained Quadratic Program (QCQP) and Second Order Cone Programming (SOCP).
Before considering QCQP and SOCP sparsification, we make a simple observation on LP sparsification, which will be used in QCQP and SOCP sparsification. Consider the following LP:
\begin{equation} \label{def:lp}
\begin{split}
    \underset{x\in \mathbb{R}^n}{\textrm{minimize}} \hspace{2mm} & c^Tx \\
    \textrm{subject to}  \hspace{2mm} & Ax = b; \\
                      & x \geq 0, \\
\end{split}
\end{equation}
where $A \in \mathbb{R}^{m \times n}, b \in \mathbb{R}^m$, and $c \in \mathbb{R}^n$. Given a solution $y$ to \eqref{def:lp}, there is an efficient algorithm to find another solution $x^*$ to \eqref{def:lp}, whose cardinality is at most $m$ \cite{Cara1907, LemonSoYe2015, Ye2008}. Now, we observe that the same result holds without the condition $x \geq 0$.

\begin{lemma} \label{lem:lp_sp}
Let $A \in \mathbb{R}^{m \times n}, b \in \mathbb{R}^m$, and $c \in \mathbb{R}^n$. Suppose that the following LP
\begin{equation} \label{def:lp2}
\begin{split}
    \underset{x\in \mathbb{R}^n}{\operatorname{minimize}} \hspace{2mm} & c^Tx \\
    \operatorname{subject \hspace{1mm} to}  \hspace{2mm} & Ax = b;
\end{split}
\end{equation}
has a finite optimal value (i.e. the problem is feasible and the objective is bounded). Then, there exists a solution $x^*$ to \eqref{def:lp2} such that $\card(x^*) \leq m$. Moreover, $x^*$ can be found in polynomial time.
\end{lemma}

\begin{proof}
Let $y$ be a solution to \eqref{def:lp2}. Let $S = \{i \in [n]: y_i<0\}$. Let $c = (c_1,\dots,c_n)^T$ and $A = (a_1,\dots,a_n)$, where $a_i \in \mathbb{R}^m$. Let $\Tilde{c} \in \mathbb{R}^n$ be defined as
\[
\Tilde{c}_i =
\begin{cases}
     c_i & \quad \textrm{if} \quad i \notin S; \\
     -c_i & \quad \textrm{if} \quad i \in S.
\end{cases}
\]
Similarly, let $\Tilde{A} = (\Tilde{a}_1, \dots, \Tilde{a}_n)$ be defined as
\[
\Tilde{a}_i =
\begin{cases}
     a_i & \quad \textrm{if} \quad i \notin S; \\
     -a_i & \quad \textrm{if} \quad i \in S.
\end{cases}
\]
Now, consider the following LP:
\begin{equation} \label{def:lp3}
\begin{split}
    \underset{x\in \mathbb{R}^n}{\textrm{minimize}} \hspace{2mm} & \Tilde{c}^T x \\
    \textrm{subject to}  \hspace{2mm} & \Tilde{A}x = b; \\
    & x \geq 0.
\end{split}
\end{equation}
Let $\Tilde{y} = |y|$, where $|\cdot|$ is applied entry-wise. We claim that $\Tilde{y}$ is a solution to \eqref{def:lp3}. To see this, note that $\Tilde{A}\Tilde{y} = Ay = b$ by definition. Suppose that there exists a solution $\Tilde{z} \in \mathbb{R}^n$ to \eqref{def:lp3} such that $\Tilde{c}^T \Tilde{z}<\Tilde{c}^T\Tilde{y}$. Let $z \in \mathbb{R}^n$ be defined as
\[
z_i =
\begin{cases}
     \Tilde{z}_i & \quad \textrm{if} \quad i \notin S; \\
     -\Tilde{z}_i & \quad \textrm{if} \quad i \in S.
\end{cases}
\]
Then, $Az = \Tilde{A}\Tilde{z} = b$ and $c^Tz = \Tilde{c}^T\Tilde{z}<\Tilde{c}^T\Tilde{y} = c^Ty$, which contradicts the fact that $y$ is a solution to \eqref{def:lp2}. Thus, $\Tilde{y}$ is a solution to \eqref{def:lp3}.

Now, we apply the LP sparsification algorithm to $\Tilde{y}$ to get a solution $\Tilde{x}$ of \eqref{def:lp3} such that $\card(\Tilde{x}) \leq m$. Let $x^* \in \mathbb{R}^n$ be defined as
\[
x^*_i =
\begin{cases}
     \Tilde{x}_i & \quad \textrm{if} \quad i \notin S; \\
     -\Tilde{x}_i & \quad \textrm{if} \quad i \in S.
\end{cases}
\]
Then, $Ax^* = \Tilde{A}\Tilde{x} = b$ and $c^Tx^* = \Tilde{c}^T\Tilde{x} = \Tilde{c}^T\Tilde{y} = c^Ty$. Thus, $x^*$ is a solution to \eqref{def:lp2}. Note that $\card(x^*) = \card(\Tilde{x}) \leq m$.
\end{proof}

\subsection{QCQP sparsification}
Consider the following Quadratically Constrained Quadratic Program (QCQP):
\begin{equation*}
\begin{split}
    \underset{x \in \mathbb{R}^n}{\textrm{minimize}} \hspace{2mm} & x^TQ_0x + c_0^Tx \\
    \textrm{subject to}  \hspace{2mm} & x^TQ_ix + c_i^Tx + d_i \leq 0, \hspace{2mm} i = 1,\dots k; \\
    & Ax = b,
\end{split}
\end{equation*}
where $Q_i \in \mathbb{S}^n_+$, $c_i \in \mathbb{R}^n$ for each $i = 0,1,\dots,k$, $d_i \in \mathbb{R}^n$ for each $i = 1,\dots,k$, $A \in \mathbb{R}^{m \times n}$, and $b \in \mathbb{R}^m$. In this section, we first give a sparsification result on QCQP. Then, we show that this sparsification result cannot be improved without additional assumptions.

\begin{theorem} \label{thm:qcqp:sparsity}
Let $Q_i \in \mathbb{S}^n_+$, $c_i \in \mathbb{R}^n$ for each $i = 0,1,\dots,k$, $d_i \in \mathbb{R}^n$ for each $i = 1,\dots,k$, $A \in \mathbb{R}^{m \times n}$, and $b \in \mathbb{R}^m$ be given.
Suppose the following QCQP:
\begin{equation} \label{qcqp:sp:def}
\begin{split}
    \underset{x \in \mathbb{R}^n}{\textrm{minimize}} \hspace{2mm} & x^TQ_0x + c_0^Tx \\
    \textrm{subject to}  \hspace{2mm} & x^TQ_ix + c_i^Tx + d_i \leq 0, \hspace{2mm} i = 1,\dots k; \\
    & Ax = b,
\end{split}
\end{equation}
is feasible. Then there exists a solution $x^*$ to \ref{qcqp:sp:def} such that 
\[\card(x^*)\leq m - 1 + \sum_{i = 0}^k (\rank(Q_i)+1).\]
Moreover, $x^*$ can be found in polynomial time.
\end{theorem}

\begin{proof}
For each $i = 0,1,\dots,k$, let $r_i = \rank(Q_i)$. Since $Q_i \in \mathbb{S}^n_+$, there exist an orthogonal matrix $U_i$ and a diagonal matrix $D_i$ such that $Q_i = U_i^T D_i U_i$ and 
\[
D_i = 
\begin{bmatrix}
B_i^2 & 0 \\
0 & 0
\end{bmatrix},
\]
where $B_i \in \mathbb{R}^{r_i \times r_i}$ is a diagonal matrix. Let
\[
P_i = \begin{bmatrix}
  B_i & 0
\end{bmatrix} U_i \in \mathbb{R}^{r_i \times n}.
\]
Then, $Q_i = P_i^T P_i$ and $x^T Q_i x = \Vert P_i x \rVert_2^2$. Let $y \in \mathbb{R}^n$ be a solution to the QCQP \eqref{qcqp:sp:def}. Now, consider the following LP:
\begin{equation} \label{pf:sp:qcqp:01}
\begin{split}
    \underset{x \in \mathbb{R}^n}{\textrm{minimize}} \hspace{2mm} & c_0^T x \\
    \textrm{subject to}  \hspace{2mm} & P_i x = P_i y, \hspace{2mm} i = 0,\dots,k; \\
                      & c_i^T x = c_i^T y, \hspace{2mm} i = 1,\dots,k; \\
                      & Ax = b.
\end{split}
\end{equation}
Note that the above LP has a finite optimal value and $y$ is a solution to it. To see this, first $y$ is clearly feasible. Second, if there is a solution $z$ such that $c_0^T z < c_0^T y$, then $z$ is a feasible point of the QCQP \eqref{qcqp:sp:def} and $z^T Q_i z + c_0^T z = \Vert P_i z \rVert_2^2 + c_0^T z < \Vert P_i y \rVert_2^2 + c_0^T y = y^T Q_i y+ c_0^T y$, which contradicts the fact that $y$ is a solution to the QCQP \eqref{qcqp:sp:def}. Thus, $y$ is a solution to the LP \eqref{pf:sp:qcqp:01}. Now, we can find a solution $x^*$ to the LP \eqref{pf:sp:qcqp:01} of cardinality at most
\[m - 1 + \sum_{i = 0}^k (\rank(Q_i)+1),\]
by Lemma \ref{lem:lp_sp}. Since $x^*$ and $y$ are both solution of \eqref{pf:sp:qcqp:01},
\[
c_0^T x^* = c_0^T y.
\]
Thus, 
\[
x{^*}^T Q_0 x^* + c_0^T x^* = \Vert P_0 x^* \rVert_2^2 + c_0^T y =  \Vert P_0 y \rVert_2^2 + c_0^T y = y^T Q_0 y + c_0^T y.
\]
Thus, $x^*$ is a solution to the QCQP \eqref{qcqp:sp:def}. Since LP sparsification can be done in polynomial time, we can find $x^*$ in polynomial time by the above procedure.
\end{proof}

Theorem \ref{thm:qcqp:sparsity} gives an upper bound on the minimal cardinality of solutions of QCQP. In the following example, we show that this bound is tight, which means that it cannot be improved without additional assumptions.

\begin{example}
Let $n,m,\rank(Q_0),\dots,\rank(Q_k) \in \mathbb{R}$ be given. To make the bound in Theorem \ref{thm:qcqp:sparsity} nontrivial, assume that
\[
m - 1 + \sum_{i = 0}^k (\rank(Q_i)+1) < n.
\]
For each $i = 0,1,\dots,k$, let
\[
r_i = \rank(Q_i), \quad \textrm{and} \quad s_i = \sum_{j = 0}^{i-1}(r_j+1).
\]
Note that $s_0 = 0$ since it is an empty sum. 
For each $i = 0,1,\dots,k$, let
\[
Q_i =
\begin{bmatrix}
0_{s_i \times s_i} & 0_{s_i \times r_i} & 0_{s_i \times (n - r_i - s_i)} \\
0_{r_i \times s_i} & I_{r_i} & 0_{r_i \times (n - r_i - s_i)} \\
0_{(n - r_i - s_i) \times s_i} & 0_{(n - r_i - s_i) \times r_i} & 0_{(n - r_i - s_i) \times (n - r_i - s_i)}
\end{bmatrix} \in \mathbb{R}^{n \times n},
\]
where $I_{r_i} \in \mathbb{R}^{r_i \times r_i}$ is the identity matrix and 
$0_{s,t} \in \mathbb{R}^{s \times t}$ denotes a matrix whose entries are zeros.
For each $i = 1,\dots,k$, let
\[
c_i = -e_{s_i + r_i + 1} - 2 \sum_{j = s_i+1}^{s_i+r_i} e_j,
\]
where $e_{s_i + r_i + 1} = (0,\dots,0,1,0,\dots,0)^T \in \mathbb{R}^n$ is the $s_i + r_i + 1$th standard basis vector. Let 
\[
c_0 = -2\sum_{i = 1}^{r_0}e_i + \sum_{i = 1}^k e_{s_i + r_i + 1}.
\]
Let
\[
A = 
\begin{bmatrix}
0_{m \times s_{k+1}} & I_{m} & 0_{m \times (n-m-s_{k+1})}
\end{bmatrix}.
\]
Consider the following QCQP:
\begin{equation} \label{eg:sp:qcqp:01}
\begin{split}
    \underset{x \in \mathbb{R}^n}{\textrm{minimize}} \hspace{2mm} & x^T Q_0 x + c_0^T x \\
    \textrm{subject to}  \hspace{2mm} & x^T Q_i x + c_i^T x + r_i + 1 \leq 0, \hspace{2mm} i = 1,\dots k; \\
    & Ax = \mathbbm{1}_m,
\end{split}
\end{equation}
where $\mathbbm{1}_m \in \mathbb{R}^m$ is a vector whose entries are ones. We claim that any solution to the above QCQP has at least $m - 1 + \sum_{i = 0}^k (\rank(Q_i)+1)$ nonzero entries. Let $z = (z_1,\dots,z_n) \in \mathbb{R}^n$ be a feasible point of the QCQP \eqref{eg:sp:qcqp:01}. Then, for each $i = 1,\dots,k$,
\begin{equation*}
    z^T Q_i z + c_i^T z + r_i + 1 \leq 0,
\end{equation*}
which implies
\begin{equation*}
    \sum_{j = s_i+1}^{s_i+r_i} z_j^2 - 2 \sum_{j = s_i+1}^{s_i+r_i} z_j - z_{s_i+r_i+1} + r_i + 1 \leq 0,
\end{equation*}
which implies
\begin{equation} \label{pf:eg:qcqp:sp:01}
    z_{s_i+r_i+1} \geq 1 + \sum_{j = s_i+1}^{s_i+r_i} (z_j^2 - 2 z_j + 1) \geq 1.
\end{equation}
Then,
\begin{equation} \label{pf:eg:qcqp:sp:02}
    z^T Q_0 z + c_0^T z = \sum_{i = 1}^{r_0} (z_i^2 - 2 z_i) + \sum_{i = 1}^k z_{s_i+r_i+1} \geq -r_0 + \sum_{i = 1}^k 1 = k - r_0,
\end{equation}
where the last step follows from equation \eqref{pf:eg:qcqp:sp:01}.
Let $x^* = (1,1,\dots,1,0,\dots,0) \in \mathbb{R}^n$ be a vector whose first $m - 1 + \sum_{i = 0}^k (r_i+1)$ entries are ones and the rest are zeros. Then, $x^*$ satisfies all the constraints in the QCQP \eqref{eg:sp:qcqp:01} and
\begin{equation} \label{pf:eg:qcqp:sp:03}
    x{^*}^T Q_0 x^* + c_0^T x^* = k - r_0.
\end{equation}
Thus, by equations \eqref{pf:eg:qcqp:sp:02} and \eqref{pf:eg:qcqp:sp:03}, the optimal value of the QCQP \eqref{eg:sp:qcqp:01} is $k - r_0$.

Now, let $y \in \mathbb{R}^n$ be a solution to the QCQP \eqref{eg:sp:qcqp:01}. We will show that $\card(y) \geq m - 1 + \sum_{i = 0}^k (\rank(Q_i)+1)$. By equation \eqref{pf:eg:qcqp:sp:01}, we have
\begin{equation*}
    y_{s_i+r_i+1} \geq 1,
\end{equation*}
for all $i = 1,\dots,k$. Since $y$ is optimal, we have
\begin{equation*}
    k - r_0 = y^T Q_0 y + c_0^T y = \sum_{i = 1}^{r_0} (y_i^2 - 2y_i) + \sum_{i = 1}^k y_{s_i+r_i+1} \geq -r_0 + \sum_{i = 1}^k 1 = k - r_0,
\end{equation*}
which implies that 
\begin{equation} \label{pf:eg:qcqp:sp:04}
    y_{s_i+r_i+1} = 1,
\end{equation}
for all $i = 1,\dots,k$ and
\begin{equation} \label{pf:eg:qcqp:sp:05}
    y_i = 1,
\end{equation}
for all $i = 1,\dots,r_0$. By equations \eqref{pf:eg:qcqp:sp:04} and \eqref{pf:eg:qcqp:sp:01},
\begin{equation} \label{pf:eg:qcqp:sp:06}
    y_j = 1
\end{equation}
for all $j = s_i + 1, \dots, s_i + r_i$, for all $i = 1,\dots,k$. Then, equations \eqref{pf:eg:qcqp:sp:04}, \eqref{pf:eg:qcqp:sp:05}, \eqref{pf:eg:qcqp:sp:06}, together with the fact that $A y = \mathbbm{1}_m$ implies that
\[
y_i = 1
\]
for all $i = 1,\dots,m - 1 + \sum_{i = 0}^k (\rank(Q_i)+1)$. Thus,
\begin{equation*}
    \card(y) \geq m - 1 + \sum_{i = 0}^k (\rank(Q_i)+1).
\end{equation*}
This shows that our bound on Theorem \ref{thm:qcqp:sparsity} is tight.
\end{example}

\subsection{SOCP sparsification}
Consider the following Second Order Cone Programming (SOCP):
\begin{equation*}
\begin{split}
    \underset{x\in \mathbb{R}^n}{\textrm{minimize}} \hspace{2mm} & c^Tx \\
    \textrm{subject to}  \hspace{2mm} & \Vert A_ix + b_i \rVert_2 \leq c_i^Tx + d_i, \hspace{2mm} i = 1,\dots,k; \\
                      & Fx = g,
\end{split}
\end{equation*}
where $A_i \in \mathbb{R}^{m_i \times n}, b_i \in \mathbb{R}^{m_i}, c_i \in \mathbb{R}^n$, and $d_i \in \mathbb{R}$ for each $i = 1,\dots,k$, $F \in \mathbb{R}^{m \times n}, c \in \mathbb{R}^n$, and $g \in \mathbb{R}^{m}$.
%Let $x^*$ be a solution to this problem.
%We now state our sparsification result.
In this section, we first give a sparsification result on SOCP. Then, we show that this sparsification result cannot be improved without additional assumptions.

\begin{theorem} \label{thm:socp:sparsity}
Let $A_i \in \mathbb{R}^{m_i \times n}, b_i \in \mathbb{R}^{m_i},  c_i \in \mathbb{R}^n$, and $d_i \in \mathbb{R}$ for each $i = 1,\dots,k$, $F \in \mathbb{R}^{m \times n}, c \in \mathbb{R}^n$, and $g \in \mathbb{R}^{m}$ be given.
Suppose the following SOCP:
\begin{equation} \label{def:socp}
\begin{split}
    \underset{x\in \mathbb{R}^n}{\textrm{minimize}} \hspace{2mm} & c^Tx \\
    \textrm{subject to}  \hspace{2mm} & \Vert A_ix + b_i \rVert_2 \leq c_i^Tx + d_i, \hspace{2mm} i = 1,\dots,k; \\
                      & Fx = g,
\end{split}
\end{equation}
is feasible. Then there exists a solution $x^*$ to \ref{def:socp} such that
\[
\card(x^*)\leq m + \sum_{i = 1}^k(m_i + 1).
\]
Moreover, $x^*$ can be found in polynomial time.
\end{theorem}
\begin{proof}
Let $y$ be a solution to the SOCP \eqref{def:socp}.
Then, consider the following LP:
\begin{equation} \label{pf:thm:socp:sp}
\begin{split}
    \underset{x \in \mathbb{R}^n}{\textrm{minimize}} \hspace{2mm} & c^Tx \\
    \textrm{subject to}  \hspace{2mm} & A_ix = A_iy, \hspace{2mm} i = 1,\dots,k; \\
                      & c_i^Tx = c_i^Ty, \hspace{2mm} i = 1,\dots,k; \\
                      & Fx = g.
\end{split}
\end{equation}
Note that the above LP has a finite optimal value and $y$ is a solution to it. To see this, first $y$ is clearly feasible. Second, if there is a solution $z$ such that $c^T z < c^T y$, then $z$ is a feasible point of the SOCP \eqref{def:socp} and $c^T z < c^T y$, which contradicts the fact that $y$ is a solution to the SOCP \eqref{def:socp}. Thus, $y$ is a solution to the LP \eqref{pf:thm:socp:sp}.
Now, we can find a solution $x^*$ to the LP \eqref{pf:thm:socp:sp} of cardinality at most 
\[m + \sum_{i = 1}^k(m_i + 1),\] 
by Lemma \ref{lem:lp_sp}. 
Since $x^*$ and $y$ are both solutions of \eqref{pf:thm:socp:sp},
\[
c^Tx^* = c^Ty.
\]
%Since $y$ satisfies all the constraints of \eqref{pf:thm:socp:sp}, 
%\[c^Tx^* \leq c^Ty.\]
%On the other hand, since $x^*$ satisfies all the constraints of \eqref{pf:thm:socp:sp}, it also satisfies all the constraints in the original SOCP \eqref{def:socp}. Thus, 
%\[c^Ty \leq c^Tx^*.\]
Thus, $x^*$ is a solution to the SOCP \eqref{def:socp}. Since LP sparsification can be done in polynomial time, we can find $x^*$ in polynomial time by the above procedure.
\end{proof}

Theorem \ref{thm:socp:sparsity} gives an upper bound on the minimal cardinality of solutions of SOCP. In the following example, we show that this bound is tight, which means that it cannot be improved without additional assumptions.

\begin{example}
Let $n,m,m_1,m_2,\dots,m_k \in \mathbb{R}$ be given. To make the bound in Theorem \ref{thm:socp:sparsity} nontrivial, assume that
\[m + \sum_{i = 1}^k (m_i+1) < n.\]
For each $i = 1,\dots,k+1$, let 
\[r_i = \sum_{j = 1}^{i-1} (m_j+1). \]
Note that $r_1 = 0$ since it is an empty sum.
For each $i = 1,\dots,k$, let 
\[c_i = e_{r_i + m_i + 1}\] and
\begin{equation*}
    E_i = 
    \begin{bmatrix}
    0_{m_i \times r_i} & I_{m_i} & 0_{m_i \times (n-r_{i+1}+1)}
    \end{bmatrix} \in \mathbb{R}^{m_i \times n}, 
\end{equation*}
where $I_{m_i} \in \mathbb{R}^{m_i \times m_i}$ is the identity matrix, $0_{s,t} \in \mathbb{R}^{s \times t}$ denotes a matrix whose entries are zeros, and $e_{r_i + m_i + 1} = (0,\dots,0,1,0,\dots,0)^T \in \mathbb{R}^n$ is the $r_i+m_i+1$th standard basis vector.
Let
\[
F = 
\begin{bmatrix}
0_{m \times r_{k+1}} & I_{m} & 0_{m \times (n-m-r_{k+1})}
\end{bmatrix}.
\]
Let 
\[
c = \sum_{i = 1}^k c_i.
\]
Consider the following SOCP:
\begin{equation} \label{socp:sp:tight:eg}
\begin{split}
    \underset{x \in \mathbb{R}^n}{\textrm{minimize}} \hspace{2mm} & c^Tx \\
    \textrm{subject to}  \hspace{2mm} & \Vert E_i x - \mathbbm{1}_{m_i} \rVert_2 \leq c_i^Tx - 1,\hspace{2mm} i = 1, \dots, k; \\
                      & Fx = \mathbbm{1}_m,
\end{split}
\end{equation}
where $\mathbbm{1}_d \in \mathbb{R}^d$ is a vector whose entries are ones.
We claim that any solution to the above SOCP has at least $m + \sum_{i = 1}^k (m_i+1)$ nonzero entries. Let $z$ be a feasible point of the SOCP \eqref{socp:sp:tight:eg}. Then, for each $i = 1, \dots, k$,
\begin{equation*}
    z_{r_i+m_i+1} - 1 = c_i^Tz -1 \geq \Vert E_i z - \mathbbm{1}_{m_i} \rVert_2 \geq 0,
\end{equation*}
which implies that
\begin{equation} \label{pf:socp:eg:03}
    z_{r_i+m_i+1} \geq 1
\end{equation}
for all $i = 1, \dots, k$. Thus,
\begin{equation} \label{pf:socp:eg:01}
    c^Tz = \sum_{i = 1}^k z_{r_i+m_i+1} \geq k.
\end{equation}
Let $x^* = (1,1,\dots,1,0,\dots,0) \in \mathbb{R}^n$ be a vector whose first $m + \sum_{i = 1}^k (m_i+1)$ entries are ones and the rest are zeros. Then, $x^*$ satisfies all the constraints in the SOCP \eqref{socp:sp:tight:eg} and
\begin{equation} \label{pf:socp:eg:02}
    c^Tx^* = k.
\end{equation}
Thus, by equations \eqref{pf:socp:eg:01} and \eqref{pf:socp:eg:02}, the optimal value of the SOCP \eqref{socp:sp:tight:eg} is $k$.

Now, let $y \in \mathbb{R}^n$ be a solution to the SOCP \eqref{socp:sp:tight:eg}. We will show that $\card(y) \geq m + \sum_{i = 1}^k (m_i+1)$. By equation \eqref{pf:socp:eg:03}, we have
\begin{equation} \label{pf:socp:eg:04}
    y_{r_i+m_i+1} \geq 1,
\end{equation}
for all $i = 1,\dots,k$.
Since $y$ is optimal, we have
\begin{equation} \label{pf:socp:eg:05}
    k = c^T y = \sum_{i = 1}^k y_{r_i+m_i+1}.
\end{equation}
By equations \eqref{pf:socp:eg:04} and \eqref{pf:socp:eg:05},
\begin{equation} \label{pf:socp:eg:06}
    y_{r_i+m_i+1} = 1,
\end{equation}
for all $i = 1,\dots,k$.
This implies that for each $i = 1,\dots,k$,
\begin{equation}
    \Vert E_iy - \mathbbm{1}_{m_i} \rVert_2 = 0,
\end{equation}
which implies that
\begin{equation} \label{pf:socp:eg:07}
    E_i y = \mathbbm{1}_{m_i}.
\end{equation}
Then, equations \eqref{pf:socp:eg:06} and \eqref{pf:socp:eg:07}, together with the fact that $F y = \mathbbm{1}_m$, imply that
\begin{equation*}
    y_i = 1
\end{equation*}
for all $i = 1,\dots,m+\sum_{j = 1}^k (m_j+1)$. Thus,
\begin{equation*}
    \card(y) \geq m+\sum_{j = 1}^k (m_j+1).
\end{equation*}
This implies that our bound in Theorem \ref{thm:socp:sparsity} is tight.
\end{example}

%This in particular gives a way to do sparsification for robust LP and LP with probabilistic constraints, which are examples of SOCP.

\section{Rank-Constrained Hyperbolic Programming}
In this section, we consider rank-constrained Hyperbolic Programming (HP), which unifies rank-constrained SDP and sparsity-constrained LP. We extend rank reduction techniques to rank-constrained QCQP and rank-constrained SOCP, which are special cases of rank-constrained HP. In addition, we study the complexity of these two optimization problems.

We first recall the definition of a hyperbolic polynomial \cite{Brand2011, hprank2, Renegar2006}.

\begin{definition}
A homogeneous polynomial $p: \mathbb{R}^n \longrightarrow \mathbb{R}$ is hyperbolic if there exists a direction $e \in \mathbb{R}^n$ such that $p(e) \neq 0$ and for each $x \in \mathbb{R}^n$ the univariate polynomial $t \mapsto p(x-te)$ has only real roots. The polynomial $p$ is said to be hyperbolic in direction $e$.
\end{definition}

Then, we recall the definition of characteristic polynomial and eigenvalues in HP \cite{Brand2011, hprank2, Renegar2006}.

\begin{definition}
Given $x \in \mathbb{R}^n$ and a polynomial $p: \mathbb{R}^n \longrightarrow \mathbb{R}$ that is hyperbolic in direction $e \in \mathbb{R}^n$, the characteristic polynomial of $x$ with respect to $p$ in direction $e$ is the univariate polynomial $\lambda \mapsto p(x - \lambda e)$. The roots of the characteristic polynomial are the eigenvalues of $x$.
\end{definition}

Next, we recall the definition of hyperbolic programming \cite{Brand2011, hprank2, Renegar2006}.

\begin{definition}
Given a hyperbolic polynomial $p: \mathbb{R}^n \longrightarrow \mathbb{R}$ that is hyperbolic in direction $e \in \mathbb{R}^n$, the hyperbolic cone for $p$ in direction $e$ is defined as
\begin{equation*}
    \Lambda_{++} \vcentcolon = \{x \in \mathbb{R}^n:\lambda_{\textrm{min}}(x) > 0 \},
\end{equation*}
where $\lambda_{\textrm{min}}(x)$ is the minimum eigenvalue of $x$. Let
\begin{equation*}
    \Lambda_+ \vcentcolon = \{x \in \mathbb{R}^n:\lambda_{\textrm{min}}(x) \geq 0 \}
\end{equation*}
be the closure of $\Lambda_{++}$.
\end{definition}

\begin{definition}
Given a hyperbolic polynomial $p: \mathbb{R}^n \longrightarrow \mathbb{R}$ that is hyperbolic in direction $e \in \mathbb{R}^n$, a hyperbolic program is an optimization problem of the form
\begin{equation} \label{def:hp}
\begin{split}
    \underset{x\in \mathbb{R}^n}{\textrm{minimize}} \hspace{2mm} & c^Tx \\
    \textrm{subject to}  \hspace{2mm} & Ax = b; \\
                      & x \in \Lambda_+,
\end{split}
\end{equation}
where $A \in \mathbb{R}^{m \times n}, b \in \mathbb{R}^m$, and $c \in \mathbb{R}^n$ are given.
\end{definition}

Finally, we recall the definition of rank in HP \cite{Brand2011, hprank2, Renegar2006}.

\begin{definition}
Let $p$ be a hyperbolic polynomial that is hyperbolic in direction $e \in \mathbb{R}^n$. The rank of $x \in \mathbb{R}^n$ is defined as $\rank(x) \vcentcolon = \deg p(e + tx)$, where $t$ is the indeterminate. Equivalently, $\rank(x)$ is the number of non-zero eigenvalues of $x$.
\end{definition}

Note that HP includes SDP and LP as special cases. Moreover, rank-constrained HP includes rank-constrained SDP and sparsity-constrained LP as special cases. To be specific, when $p(X) = \det(X)$ and $e = I$ (in this case, the domain of $p$ is the set of symmetric matrices which can be identified as $\mathbb{R}^{n(n+1)/2}$), HP becomes SDP and $\rank(X)$ is simply the usual rank of a matrix \cite{Brand2011, hprank2, Renegar2006}. In addition, when $p(x) = \prod_{i = 1}^n x_i$, where $x_i$ is the $i$th component of $x$ and $e = (1,1,\dots, 1)$, HP becomes LP and $\rank(x) = \card(x)$ \cite{Brand2011, hprank2, Renegar2006}.

\subsection{Rank-Constrained SOCP}
In this section, we study rank-constrained SOCP. We first define the rank of SOCP by viewing it as a HP. Then, we give a rank reduction result for SOCP. Next, we show that Max-Cut can be written as a rank-constrained SOCP. Finally, we study the complexity of rank-constrained SOCP and show that it is NP-hard in certain circumstances.

\subsubsection{SOCP rank reduction}
We consider the following SOCP:
\begin{equation} \label{hp:socp:def}
\begin{split}
    \underset{x\in \mathbb{R}^n}{\textrm{minimize}} \hspace{2mm} & c^Tx \\
    \textrm{subject to}  \hspace{2mm} & \Vert A_ix + b_i \rVert_2 \leq c_i^Tx + d_i, \hspace{2mm} i = 1,\dots,k; \\
                      & Fx = g,
\end{split}
\end{equation}
where $A_i \in \mathbb{R}^{m_i \times n}, b_i \in \mathbb{R}^{m_i}, c_i \in \mathbb{R}^n$, and $d_i \in \mathbb{R}$ for each $i = 1,\dots,k$, $F \in \mathbb{R}^{m \times n}, c \in \mathbb{R}^n$, and $g \in \mathbb{R}^{m}$. We first show that it can be written as a HP. This step is similar to the Lorentz cone example in \cite{Sendov2001}. First, we associate each second order cone constraint $\Vert A_ix + b_i \rVert_2 \leq c_i^T x + d_i$ with a new variable $y_i$. Let $y = (y_1, \dots, y_k)$. Let $z = (x,y)$. For each $j \in [m_i]$, let $a_{i,j}^T$ be the $j$th row of $A_i$, $b_{i,j}$ be the $j$th entry of $b_i$, and let $q_{i,j}(x) = a_{i,j}^T x + b_{i,j}$. Then let 
 \[
 p_i(z) = y_i^2 - \sum_{j = 1}^{m_i} q_{i,j}(x)^2.
 \]
Now let
\[
p(z) = \prod_{i = 1}^k p_i(z).
\]
Let $e = (0,\dots,0,1,\dots,1)$ where there are $n$ zeros followed by $k$ ones. For each $i$,
\begin{equation} \label{eq:2}
p_i(z - te) = (y_i - t)^2 - \sum_{j = 1}^{m_i} q_{i,j}(x)^2 = t^2 - 2y_it + \Biggl(y_i^2 - \sum_{j = 1}^{m_i} q_{i,j}(x)^2\Biggr).
\end{equation}
The discriminant
\[
\Delta_i = 4y_i^2 - 4\Biggl(y_i^2 - \sum_{j = 1}^{m_i} q_{i,j}(x)^2\Biggr) = 4\sum_{j = 1}^{m_i} q_{i,j}(x)^2 \geq 0.
\]
Thus, $p_i$ is hyperbolic in $e$ for all $i$. Hence, $p$ is hyperbolic in $e$. 
%Note that $p(z - te)$ only has positive roots if and only if $p_i(z - te)$ only has positive roots for all $i$. 
Note that roots of $p(z - te)$ are positive if and only if roots of $p_i(z - te)$ are positive for all $i$.
From equation (\ref{eq:2}), we see that this holds if and only if
\[
y_i \geq 0 \quad \textrm{and} \quad y_i^2 \geq \sum_{j = 1}^{m_i} q_{i,j}(x)^2.
\]
This is equivalent to
\[
y_i \geq \sqrt{\sum_{j = 1}^{m_i} q_{i,j}(x)^2} = \Vert A_ix + b_i \rVert_2.
\]
Now for each $i$, we add the linear constraint
\[
y_i = c_i^Tx + d_i.
\]
Then the resulting HP with the original linear constraint $Fx = g$ and the new linear constraints $y_i = c_i^Tx + d_1$ is equivalent to the SOCP problem \eqref{hp:socp:def}. 

Next, we consider the rank. Note that
\[
p_i(e + tz) = (ty_i + 1)^2 - \sum_{j = 1}^{m_i} q_{i,j}(x)^2 t^2 = \Biggl(y_i^2 - \sum_{j = 1}^{m_i} q_{i,j}(x)^2\Biggr) t^2 + 2y_it + 1.
\]
Then we have
\[
\deg(p_i) = \begin{cases}
     0 & \hspace{2mm} \textrm{if} \hspace{2mm} y_i^2 = \sum_{j = 1}^{m_i} q_{i,j}(x)^2 = 0 \\
     1 & \hspace{2mm} \textrm{if} \hspace{2mm} y_i^2 = \sum_{j = 1}^{m_i} q_{i,j}(x)^2 \neq 0 \\
     2 &  \hspace{2mm} \textrm{otherwise}.
\end{cases}
\]
Let $s(x)$ be the number of second order cone constraints that are satisfied with equality. Let $e(x)$ be the number of second order cone constraints that are satisfied with equality and both sides of the constraints are zero. 
%Then since the degree of $p$ is the sum of the degrees of $p_i$, 
Since
\[
\deg(p(e+tz)) = \sum_{i = 1}^k \deg(p_i(e + tz)),
\]
we have
\[
\rank(z) = 2k - s(x) - e(x).
\]
In addition, we define 
\[\rank(x) = \rank(z).\]

%Before stating our rank reduction result, we give some definition. To make the result cleaner, we assume that
%\[
%m_1 = m_2 = \dots = m_k = m.
%\]

Now we consider SOCP rank reduction. 
%For simplicity, we consider the case
%\[
%_1 = m_2 = \dots = m_k = m.
%\]
%The relationship between rank and the number of constraints becomes clearer under this assumption, although a similar result can be obtained without it.

\begin{theorem} \label{socp:rank:reduction}
Let $A_i \in \mathbb{R}^{m_i \times n}, b_i \in \mathbb{R}^{m_i},  c_i \in \mathbb{R}^n$, and $d_i \in \mathbb{R}$ for each $i = 1,\dots,k$, $F \in \mathbb{R}^{m \times n}, c \in \mathbb{R}^n$, and $g \in \mathbb{R}^{m}$ be given.
Suppose the following SOCP:
\begin{equation} \label{def:socp:hp}
\begin{split}
    \underset{x\in \mathbb{R}^n}{\textrm{minimize}} \hspace{2mm} & c^Tx \\
    \textrm{subject to}  \hspace{2mm} & \Vert A_ix + b_i \rVert_2 \leq c_i^Tx + d_i, \hspace{2mm} i = 1,\dots,k; \\
                      & Fx = g,
\end{split}
\end{equation}
is feasible and 
\[
\bigcap_{i \in [k+1]} \ker B_i = \{0\},
\]
where 
\[
B_i = \begin{bmatrix}
A_{i} \\
c_{i}^T
\end{bmatrix} \quad \textrm{for each} \hspace{2mm} i \in [k], \hspace{2mm} \textrm{and} \quad B_{k+1} = \begin{bmatrix}
F \\
c^T
\end{bmatrix}.
\]
Then there exists a solution $x^*$ to \ref{def:socp:hp} such that
\[
\rank(x) \leq 2k - \Bigl \lceil \frac{n}{\max(m,m_1,m_2,\dots,m_k)+1} \Bigr \rceil + 1.
\]
Moreover, $x^*$ can be found in polynomial time.
\end{theorem}
\begin{proof}
It suffices to show that we can find a solution $x$ such that
\[
    s(x) \geq \Bigl \lceil \frac{n}{\max(m,m_1,m_2,\dots,m_k)+1} \Bigr \rceil - 1. 
\]
Let
\[
m' = \max(m,m_1,m_2,\dots,m_k).
\]
Let $x^{(0)}$ be a solution to the SOCP. We define an algorithm iteratively with the inductive hypothesis that at the beginning of step $i$,
\[
\lVert A_j x^{(i-1)} + b_j \rVert_2 = c_j^T x^{(i-1)} + d_j,
\]
for all $j \in S_{i-1}$, where $x^{(i)}$ is the value of $x$ at the end of the $i$th iteration and $S_i$ is a set of size $i$ which is updated in each iteration. This clearly holds for $i = 1$ with $S_0 = \emptyset$. At step $i$, we do the following. Let $S_{i - 1} = \{u_1, \dots, u_{i-1}\}$ and
\[
M_i = \begin{bmatrix}
B_{k+1} \\
B_{u_1} \\
\vdots \\
B_{u_{i-1}}
\end{bmatrix}
\]
Then $M_i \in \mathbb{R}^{t_i \times n}$ for some $t_i \leq (m'+1)i$. Suppose that 
\[
i \leq \bigg \lceil \frac{n}{m'+1} \bigg \rceil - 1. 
\]
Then $t_i \leq (m'+1)i<n$. Thus, $\ker M_i \neq \{0\}$. Take $v \in \ker M_i$ such that $v \neq 0$. Since
\[
\bigcap_{i \in [k+1]} \ker B_i = \{0\},
\]
there exists $u \in [k]$, such that $B_{u} v \neq 0$. Clearly, $u \notin S_{i-1}$. Now we consider two cases. If $c_{u}^T v = 0$, then $A_{u} v \neq 0$. Thus,
\[
\lVert A_{u} (x^{(i-1)} +\lambda v) + b_{u} \rVert_2 - (c_{u}^T (x^{(i-1)} +\lambda v) + d_{u}) \longrightarrow \infty,
\]
as $\lambda \longrightarrow \infty$. Since the above expression is less than or equal to $0$ when $\lambda = 0$, there exists $\lambda^*$ such that
\[
\lVert A_{u} (x^{(i-1)} +\lambda^* v) + b_{u} \rVert_2 - (c_{u}^T (x^{(i-1)} +\lambda^* v) + d_{u}) = 0.
\]
Now if $c_{u}^T v \neq 0$, then multiplying $v$ by $-1$ if necessary, we can assume that $c_{u}^T v < 0$. Thus,
\[
\lVert A_{u} (x^{(i-1)} +\lambda v) + b_{u} \rVert_2 - (c_{u}^T (x^{(i-1)} +\lambda v) + d_{u}) \longrightarrow \infty,
\]
as $\lambda \longrightarrow \infty$. Then again we have there exists $\lambda^*$ such that
\[
\lVert A_{u} (x^{(i-1)} +\lambda^* v) + b_{u} \rVert_2 - (c_{u}^T (x^{(i-1)} +\lambda^* v) + d_{u}) = 0.
\]
Now let 
\[E = \{u:B_{u} v \neq 0\}.\]
For each $u \in E$, let 
\[
\lambda_u = \arg\min_{\lambda \in \mathbb{R}} \Biggl \{|\lambda|: \lVert A_{u} (x^{(i-1)} +\lambda v) + b_{u} \rVert_2 = (c_{u}^T (x^{(i-1)} +\lambda v) + d_{u}) \Biggr \}.
\]
Now let 
\[u_i \in \arg\min_{u \in E} |\lambda_u|.\]
Update
\[
\begin{split}
    x^{(i)} & = x^{(i-1)} + \lambda_{u_i} v \\
    S_i & = S_{i-1} \cup \{u_i\}.
\end{split}
\]
Then clearly
\[
\lVert A_j x^{(i-1)} + b_j \rVert_2 = c_j^T x^{(i-1)} + d_j,
\]
for all $j \in S_{i}$ and
\[
|S_i| = |S_{i-1}| + 1 = i.
\]
Note that $x^{(i)}$ is still a feasible solution since for all $u \in E$,
\[
\lVert A_{u} (x^{(i-1)} +\lambda^* v) + b_{u} \rVert_2 - (c_{u}^T (x^{(i-1)} +\lambda^* v) + d_{u}) \leq 0,
\]
by minimality of $|\lambda_{u_i}|$. This process stop when $i = \lceil n/(m'+1)\rceil$. Then we have
\[
\lVert A_j x^{(i-1)} + b_j \rVert_2 = c_j^T x^{(i-1)} + d_j,
\]
for all $j \in S_{\lceil n/(m'+1) \rceil - 1}$ and thus
\[
s(x^{(i-1)}) \geq \Bigl \lceil \frac{n}{m'+1} \Bigr \rceil - 1.
\]
\end{proof}

\subsubsection{Complexity of rank-constrained SOCP}
%Now we give some hardness results for rank constrained SOCP. Before that, we first give an example to show that Max-Cut could be formulated as a rank constrained SOCP problem. Our example is a slight modification of the SOCP relaxation results in \cite{MST2003}.

In this section, we study the complexity of rank-constrained SOCP. First, we will show that Max-Cut could be formulated as rank-constrained SOCP. Then, we will use this reduction to show that rank-constrained SOCP is NP-hard under certain circumstances. Our reduction of Max-Cut to rank-constrained SOCP is a slight modification of the SOCP relaxation results in \cite{MST2003}.

\begin{example}
First, we consider a nonconvex Quadratically Constrained
Linear Program (QCLP):
\begin{equation} \label{def:nonconvex:qclp}
\begin{split}
    \underset{x \in \mathbb{R}^n}{\textrm{minimize}} \hspace{2mm} & c^Tx \\
    \textrm{subject to}  \hspace{2mm} & x^TQ_ix+ g_i^Tx + f_i = 0,\hspace{2mm} i = 1,\dots,m,
\end{split}
\end{equation}
where $Q_i \in \mathbb{S}^n$, $g_i \in \mathbb{R}^n$, and $f_i \in \mathbb{R}$. We will first show that the above nonconvex QCLP is equivalent to a rank-constrained SOCP. Then, since Max-Cut can be written as a nonconvex QCLP, Max-Cut is also equivalent to a rank-constrained SOCP.

\pmb{QCLP to rank-constrained SOCP:}

Since
\[
x^TQ_ix = \tr(Q_i x x^T),
\]
the QCLP in \eqref{def:nonconvex:qclp} is clearly equivalent to
\[
\begin{split}
    \underset{x \in \mathbb{R}^n, X \in \mathbb{S}^n}{\textrm{minimize}} \hspace{2mm} & c^Tx \\
    \textrm{subject to}  \hspace{2mm} & \tr(Q_i X)+ g_i^Tx + f_i = 0,\hspace{2mm} i = 1,\dots,m, \\
    & X - xx^T = 0.
\end{split}
\]
The above is an LP with the additional constraint $X - xx^T = 0$. Thus, it suffices to write that constraint as a second order cone constraint. Let $C_1, \dots, C_r$ be an orthonormal basis for $\mathbb{S}^n$. Then $X - xx^T = 0$ if and only if
\[
\tr(C_i (X - xx^T)) = 0,
\]
for all $i \in [r]$. For each $i$, there exists $\lambda_i$ such that $C_i + \lambda_i I \in \mathbb{S}^n_+$. Let $\Tilde{C}_i = C_i + \lambda_i I$. Then for all $i \in [r]$, $\tr(C_i (X - xx^T)) = 0$ if
\begin{equation} \label{nonconvex:qclp:eq:1}
    \tr(\Tilde{C}_i (X - xx^T)) = 0, \hspace{2mm} \textrm{and} \hspace{2mm} \tr(I (X - xx^T)) = 0.
\end{equation}
On the other hand, if $X - xx^T = 0$, then equation \eqref{nonconvex:qclp:eq:1} certainly holds for all $i \in [r]$. Thus, $X - xx^T = 0$ if and only if equation \eqref{nonconvex:qclp:eq:1} holds for all $i \in [r]$.

Now let $A \in \mathbb{S}^n_+$. Then $A = VV^T$ for some $V \in \mathbb{R}^{n \times n}$. Note that
\begin{equation} \label{maxcut:hp:key}
    \tr(A (X - xx^T)) = \tr(AX) - x^TVV^Tx = a - u^Tu,
\end{equation}
where $a = \tr(AX), u = V^T x$. Now note that
\[ a- u^Tu = 0\]
if and only if
\[
(a+1)^2 = (a - 1)^2 + 4u^Tu.
\]
Let $w = \begin{bmatrix}
a-1 \\
2u
\end{bmatrix}$,
then the above equation is equivalent to
\[
\lVert w \rVert_2 = a+1.
\]
In other words, we get
\begin{equation} \label{eq:30}
\Bigg \lVert 
\begin{bmatrix}
\tr(AX)-1 \\
2V^T x
\end{bmatrix}
\Bigg \rVert_2
= \tr(AX) + 1,
\end{equation}
which is a second order cone constraint with equality. Note that the left hand side and right hand side of \eqref{eq:30} cannot both be $0$. Otherwise, we would have $tr(AX) = 1$ from the left hand side and $tr(AX) = -1$ from the right hand side, which is a contradiction. Thus,
\[
e(x) = 0,
\]
for all $x$ that satisfies \eqref{eq:30}, where $e(x)$ is the number of second order cone constrained that are satisfied with equality and both sides are zero. Then, we need 
\[
s(x) = k \quad \quad \quad \quad \textrm{and} \quad \quad \quad \quad e(x) = 0,
\]
where
\[
k = 1+\frac{n(n+1)}{2}
\]
is the number of second order cone constraints and $s(x)$ is the number of second order cone constrained that are satisfied with equality. Thus, the original QCLP could be written as a SOCP problem with the additional rank constraint
\[
\rank(x) \leq k.
\]

\pmb{Max-Cut to QCLP:}
It is well known that Max-Cut can be written in the form \cite{MST2003}
\[
\begin{split}
    \underset{x \in \mathbb{R}^n}{\textrm{minimize}} \hspace{2mm} & x^TQx \\
    \textrm{subject to}  \hspace{2mm} & x_i^2 - 1 = 0,\hspace{2mm} i = 1,\dots,n,
\end{split}
\]
for some $Q \in \mathbb{S}^n$. Then it is equivalent to
\[
\begin{split}
    \underset{x \in \mathbb{R}^n}{\textrm{minimize}} \hspace{2mm} & t \\
    \textrm{subject to}  \hspace{2mm} & x_i^2 - 1 = 0,\hspace{2mm} i = 1,\dots,n, \\
    & x^TQx - t = 0,
\end{split}
\]
which is then in the form of the nonconvex QCLP in \eqref{def:nonconvex:qclp}. Thus, Max-Cut is equivalent to a rank-constrained SOCP.
\end{example}

%From the above example, we see that SOCP rank reduction provides a way to check whether SOCP relaxation of nonconvex QCQP is exact. A subclass of nonconvex QCLP has exact SOCP relaxation if and only if the corresponding SOCP can be reduced to a certain rank. This motivates the study of SOCP rank reduction.
From the above example, we see that rank-constrained SOCP includes some interesting combinatorial problems, which motivates the study of rank-constrained SOCP:
\[
\begin{split}
    \underset{x\in \mathbb{R}^n}{\textrm{minimize}} \hspace{2mm} & c^Tx \\
    \textrm{subject to}  \hspace{2mm} & \Vert A_ix + b_i \rVert_2 \leq c_i^Tx + d_i, \hspace{2mm} i = 1,\dots,k; \\
                      & Fx = g,\\
                      & \rank(x) \leq r(k)
\end{split}
\]
where $r(k)$ is a function in $k$, $A_i \in \mathbb{R}^{m_i \times n}, b_i \in \mathbb{R}^{m_i}, c_i \in \mathbb{R}^n$, and $d_i \in \mathbb{R}$ for each $i = 1,\dots,k$, $F \in \mathbb{R}^{m \times n}, c \in \mathbb{R}^n$, and $g \in \mathbb{R}^{m}$. Next, we give a complexity result of this problem.
 
\begin{theorem} \label{socp:rank:hardness}
Let $A_i \in \mathbb{R}^{m_i \times n}, b_i \in \mathbb{R}^{m_i},  c_i \in \mathbb{R}^n$, and $d_i \in \mathbb{R}$ for each $i = 1,\dots,k$, $F \in \mathbb{R}^{m \times n}, c \in \mathbb{R}^n$, and $g \in \mathbb{R}^{m}$ be given. Let $s \geq 0$ be a constant. Then the following rank-constrained SOCP:
\begin{equation*}
\begin{split}
    \underset{x\in \mathbb{R}^n}{\textrm{minimize}} \hspace{2mm} & c^Tx \\
    \textrm{subject to}  \hspace{2mm} & \Vert A_ix + b_i \rVert_2 \leq c_i^Tx + d_i, \hspace{2mm} i = 1,\dots,k; \\
                      & Fx = g,\\
                      & \rank(x) \leq k+s,
\end{split}
\end{equation*}
 is NP-hard. 
\end{theorem}

\begin{proof}
From the example at the beginning of this section, we see that rank-constrained SOCP with $r(k) = k$ is NP-hard since we can reduce Max-Cut to it. In other words, the following problem:
\begin{equation} \label{eq:10}
\begin{split}
    \underset{x\in \mathbb{R}^n}{\textrm{minimize}} \hspace{2mm} & c^Tx \\
    \textrm{subject to}  \hspace{2mm} & \Vert A_ix + b_i \rVert_2 \leq c_i^Tx + d_i, \hspace{2mm} i = 1,\dots,k; \\
                      & Fx = g,\\
                      & \rank(x) \leq k,
\end{split}
\end{equation}
is NP-hard.
Now, we show how to increase $r(k)$ from $k$ to $k+1$. Consider the following problem:
\begin{equation} \label{eq:19}
\begin{split}
    \underset{x\in \mathbb{R}^{n+1}}{\textrm{minimize}} \hspace{2mm} & \Tilde{c}^Tx \\
    \textrm{subject to}  \hspace{2mm} & \big \lVert \Tilde{A_i}x + \Tilde{b}_i \big \rVert_2 \leq \Tilde{c_i}^Tx + d_i, \hspace{2mm} i = 1,\dots,k; \\
                      & \Tilde{F}x = g,\\
                      & x_{n+1} = 1,\\
                      & \lVert Ex \rVert_2 \leq 2, \\
                      & \rank(x) \leq (k+1)+1,
\end{split}
\end{equation}
where $\Tilde{A}_i = \begin{bsmallmatrix}
                     A_i \\
                     0
\end{bsmallmatrix}, \Tilde{c} = \begin{bsmallmatrix}
                     c \\
                     0
\end{bsmallmatrix}, \Tilde{b}_i = \begin{bsmallmatrix}
                     b_i \\
                     0
\end{bsmallmatrix}, \Tilde{c}_i = \begin{bsmallmatrix}
                     c_i \\
                     0
\end{bsmallmatrix}, \Tilde{F} = \begin{bsmallmatrix}
                     F & 0
\end{bsmallmatrix}, E = e_{n+1}e_{n+1}^T$, and $e_{n+1} = (0,\dots,0,1)$. Note that since $x_{n+1} = 1$,
\[
\lVert Ex \rVert_2 = 1<2.
\]
Thus, this second order cone constraint will contribute $2$ to the rank. Thus, \eqref{eq:19} is equivalent to \eqref{eq:10}. Note that there are $k+1$ second order cone constraints in \eqref{eq:19}. Using this argument $s$ times finishes the proof.
\end{proof}

%Note that if restrict ourselves to the case when no second order cone constraints can be satisfied with each side of the constraint equals $0$, then the rank is always bounded below by $k$. In this case, we can always shift the rank by $k-1$ and use the alternative definition that $\rank(x) = \rank(z) - k + 1$. Then, theorem \ref{socp:rank:hardness} implies that rank-constrained SOCP is NP-hard for any constant $r$, which resembles the results in the SDP case.

\subsection{Rank-Constrained QCQP}
In this section, we study rank-constrained QCQP. We first define the rank of QCQP by viewing it as a SOCP. Then, we apply the results in the previous section to do rank reduction on QCQP. Next, we show that Max-Cut can also be written as a rank-constrained QCQP. Finally, we show that rank-constrained QCQP is NP-hard in certain circumstances.

\subsubsection{QCQP rank reduction}
Consider the following QCQP:
\begin{equation} \label{def:qcqp:rank}
\begin{split}
    \underset{x \in \mathbb{R}^n}{\textrm{minimize}} \hspace{2mm} & x^TQ_0x + c_0^Tx \\
    \textrm{subject to}  \hspace{2mm} & x^TQ_ix + c_i^Tx + d_i \leq 0, \hspace{2mm} i = 1,\dots k; \\
    & Ax = b,
\end{split}
\end{equation}
where $Q_i \in \mathbb{S}^n_+$, $c_i \in \mathbb{R}^n$ for each $i = 0,1,\dots,k$, $d_i \in \mathbb{R}^n$ for each $i = 1,\dots,k$, $A \in \mathbb{R}^{m \times n}$, and $b \in \mathbb{R}^m$. First, we write QCQP \ref{def:qcqp:rank} as a SOCP \ref{hp:socp:def}. To do this, we first consider the epigraph version of \ref{def:qcqp:rank}:
\begin{equation} \label{def:qcqp:rank2}
\begin{split}
    \underset{x \in \mathbb{R}^n, t \in \mathbb{R}}{\textrm{minimize}} \hspace{2mm} & t \\
    \textrm{subject to}  \hspace{2mm} 
    & x^TQ_0x + c_0^Tx - t \leq 0 \\
    & x^TQ_ix + c_i^Tx + d_i \leq 0, \hspace{2mm} i = 1,\dots k; \\
    & Ax = b.
\end{split}
\end{equation}
Then, it suffices to convert a quadratic constraint
\begin{equation} \label{hp:qcqp:eq:1}
    x^TQ_ix + c_i^Tx + d_i \leq 0
\end{equation}
to a second order cone constraint. For each $i = 0,1,\dots,k$, let $r_i = \rank(Q_i)$. Then there exists $P_i \in \mathbb{R}^{r_i \times r_i}$ such that $Q_i = P_i^T P_i$ since $Q_i \in \mathbb{S}^n_+$. Then, equation \ref{hp:qcqp:eq:1} becomes
\begin{equation} \label{hp:qcqp:eq:2}
    \lVert P_i x \rVert_2^2 \leq -c_i^T x - d_i.
\end{equation}
Since
\[
-c_i^T x - d_i = (1/4 -c_i^T x - d_i)^2 - (1/4 + c_i^T x + d_i)^2,
\]
equation \ref{hp:qcqp:eq:2} is equivalent to
\begin{equation*}
    \lVert P_i x \rVert_2^2 + (1/4 + c_i^T x + d_i)^2 \leq (1/4 -c_i^T x - d_i)^2,
\end{equation*}
which is equivalent to
\begin{equation} \label{hp:qcqp:eq:3}
    \Bigg \lVert 
    \begin{bmatrix}
    P_i x \\
    1/4 + c_i^T x + d_i
    \end{bmatrix}
    \Bigg \rVert_2
    \leq \lVert 1/4 -c_i^T x - d_i \rVert_2,
\end{equation}
which is a second order cone constraint. Note that the left hand side and right hand side of equation \ref{hp:qcqp:eq:3} cannot both be zeroes, since $1/4 + c_i^T x + d_i$ and $1/4 - c_i^T x - d_i$ cannot both be zeroes. Thus, according to the definition of rank for SOCP, the rank in QCQP \ref{def:qcqp:rank} is defined as
\begin{equation*}
    \rank(x) \vcentcolon = 2k+1 - s(x),
\end{equation*}
where $s(x)$ is the number of quadratic constraints in \ref{def:qcqp:rank} that are satisfied with equality (i.e. $x^TQ_ix + c_i^Tx + d_i = 0$). Note that by writing QCQP as SOCP, we have $k+1$ second order cone constraints. However, $x^TQ_0x + c_0^Tx - t = 0$ for any $x$ that attains the optimal value.
Note that in QCQP,
\begin{equation*}
    \rank(x) \geq k + 1
\end{equation*}
for all $x \in \mathbb{R}^n$.

Now, we do rank reduction on QCQP. By viewing QCQP as a SOCP and applying Theorem \ref{socp:rank:reduction}, we get the following result.

\begin{theorem}
Let $Q_i \in \mathbb{S}^n_+$, $c_i \in \mathbb{R}^n$ for each $i = 0,1,\dots,k$, $d_i \in \mathbb{R}^n$ for each $i = 1,\dots,k$, $A \in \mathbb{R}^{m \times n}$, and $b \in \mathbb{R}^m$ be given.
Suppose the following QCQP:
\begin{equation} \label{qcqp:sp:def2}
\begin{split}
    \underset{x \in \mathbb{R}^n}{\textrm{minimize}} \hspace{2mm} & x^TQ_0x + c_0^Tx \\
    \textrm{subject to}  \hspace{2mm} & x^TQ_ix + c_i^Tx + d_i \leq 0, \hspace{2mm} i = 1,\dots k; \\
    & Ax = b,
\end{split}
\end{equation}
is feasible and
\[
\bigcap_{i = 0}^{k+1} \ker B_i = \{0\},
\]
where 
\[
B_i = \begin{bmatrix}
Q_{i} \\
c_{i}^T
\end{bmatrix} \quad \textrm{for each} \hspace{2mm} i = 0,1,\dots,k, \hspace{2mm} \textrm{and} \quad B_{k+1} = A.
\]
Then there exists a solution $x^*$ to \ref{qcqp:sp:def2} such that 
\begin{equation*}
    \rank(x^*) \leq 2k + 3 - \Bigl \lceil \frac{n}{\max(m-1,\rank(Q_0),\rank(Q_1),\dots,\rank(Q_k))+1} \Bigr \rceil.
\end{equation*}
Moreover, $x^*$ can be found in polynomial time.
\end{theorem}

\begin{proof}
First, we write the QCQP as a SOCP. Then we have $k+1$ second order cone constraints. Note that for $B_{k+1}$, we don't need to concatenate $A$ with $c_0$ as we did in the SOCP case since the objective $x^TQ_0x + c_0^Tx$ becomes a second order cone constraint. In addition, if $Q_i = P_i^T P_i$, then $\ker(Q_i) = \ker(P_i)$. In addition, for $i = 1,\dots,k$ instead of defining $B_i$ as a concatenation of $Q_i, c_i^T$ and $c_i^T$, we just need to define it as a concatenation of $Q_i$ and $c_i^T$ since they have the same kernel. Applying Theorem \ref{socp:rank:reduction} to the resulting SOCP finishes the proof.
\end{proof}

\subsubsection{Complexity of rank-constrained QCQP}
In this section, we study the complexity of rank-constrained QCQP. First, note that Max-Cut can be written as a rank-constrained QCQP. This follows from the fact that equation \ref{maxcut:hp:key} can be seen as a quadratic constraint
\begin{equation*}
    x^TAx - \tr(AX) \leq 0,
\end{equation*}
that holds with equality. Thus, Max-Cut can be written as a rank-constrained QCQP where the constraint on rank is $\rank(x) \leq k+1$. With similar techniques as we use in Theorem \ref{socp:rank:hardness}, we obtain the following result.

\begin{theorem}
Let $Q_i \in \mathbb{S}^n_+$, $c_i \in \mathbb{R}^n$ for each $i = 0,1,\dots,k$, $d_i \in \mathbb{R}^n$ for each $i = 1,\dots,k$, $A \in \mathbb{R}^{m \times n}$, and $b \in \mathbb{R}^m$ be given. Let $s \geq 1$ be a constant. Then the following rank-constrained QCQP:
\begin{equation*}
\begin{split}
    \underset{x \in \mathbb{R}^n}{\textrm{minimize}} \hspace{2mm} 
    & x^TQ_0x + c_0^Tx \\
    \textrm{subject to}  \hspace{2mm} 
    & x^TQ_ix + c_i^Tx + d_i \leq 0, \hspace{2mm} i = 1,\dots k; \\
    & Ax = b,\\
    & \rank(x) \leq k+s,
\end{split}
\end{equation*}
 is NP-hard. 
\end{theorem}

\begin{proof}
Since we can write Max-Cut as a rank-constrained QCQP with the rank constraint $\rank(x) \leq k+1$, the following problem
\begin{equation} \label{pf:qcqp:rank:1}
\begin{split}
    \underset{x \in \mathbb{R}^n}{\textrm{minimize}} \hspace{2mm} 
    & x^TQ_0x + c_0^Tx \\
    \textrm{subject to}  \hspace{2mm} 
    & x^TQ_ix + c_i^Tx + d_i \leq 0, \hspace{2mm} i = 1,\dots k; \\
    & Ax = b,\\
    & \rank(x) \leq k+1,
\end{split}
\end{equation}
is NP-hard. Now, we show how to increase the rank constraint from $k+1$ to $k+2$. For each $i = 0,\dots,k$, $Q_i = P_i^T P_i$ for some $P_i \in \mathbb{R}^{n \times n}$. Consider the following problem:
\begin{equation} \label{pf:qcqp:rank:2}
\begin{split}
    \underset{x \in \mathbb{R}^{n+1}}{\textrm{minimize}} \hspace{2mm} 
    & x^T\Tilde{P_0}^T\Tilde{P_0}x + \Tilde{c_0}^Tx \\
    \textrm{subject to}  \hspace{2mm} 
    & x^T\Tilde{P_i}^T\Tilde{P_i}x + \Tilde{c_i}^Tx + d_i \leq 0, \hspace{2mm} i = 1,\dots k; \\
    & \Tilde{A}x = b,\\
    & x_{k+1} = 1, \\
    & x^T E x \leq 2, \\
    & \rank(x) \leq (k+1)+2,
\end{split}
\end{equation}
where $\Tilde{A} = \begin{bsmallmatrix}
                     A &
                     0
\end{bsmallmatrix}, \Tilde{P}_i = \begin{bsmallmatrix}
                     P_i &
                     0
\end{bsmallmatrix}, \Tilde{c}_i = \begin{bsmallmatrix}
                     c_i \\
                     0
\end{bsmallmatrix}, E = e_{n+1}e_{n+1}^T$, and $e_{n+1} = (0,\dots,0,1)$. Note that since $x_{n+1} = 1$,
\[
x^T E x = 1<2.
\]
Thus, this quadratic constraint will contribute $2$ to the rank. Thus, \eqref{pf:qcqp:rank:2} is equivalent to \eqref{pf:qcqp:rank:1}. Note that there are $k+1$ quadratic constraints in \eqref{pf:qcqp:rank:2}. Using this argument $s-1$ times finishes the proof.
\end{proof}

\section{conclusion}
In this paper, we study rank-constrained and sparsity-constrained HP. For rank-constrained HP, we design algorithms for rank reduction and study the complexity of rank-constrained HP. We showed that both rank-constrained QCQP and rank-constrained SOCP are NP-hard. In addition, we show that there is a phase transition in the complexity of rank-constrained SDP with $m$ linear constraints when the rank constraint $r(m)$ passes through $\sqrt{2m}$.

For sparsity-constrained HP, we extend results on LP sparsification to QCQP and SOCP and show that our results give tight upper bounds on minimal cardinality solutions to QCQP and SOCP.

\bibliographystyle{abbrv}

\bibliography{rank}

\end{document}